\documentclass[11pt]{article}
\usepackage{amsmath, amssymb, amsthm}
\usepackage{graphicx}
\usepackage{mathdots}

\date{}
\title{\vspace{-0.7cm}On the grid Ramsey problem and related questions}
\author{
David Conlon \thanks{Mathematical Institute, Oxford OX2 6GG,
United Kingdom. Email: {\tt david.conlon@maths.ox.ac.uk}. Research
supported by a Royal Society University Research Fellowship.}
\and
Jacob Fox \thanks{Department of Mathematics, MIT, Cambridge,
MA 02139-4307. Email: {\tt fox@math.mit.edu}. Research supported by
a Packard Fellowship, by a Simons Fellowship, by NSF grant DMS-1069197, by an Alfred P. Sloan Fellowship  and by an MIT NEC Corporation Award.}
\and
Choongbum Lee \thanks{Department of Mathematics,
MIT, Cambridge, MA 02139-4307. Email: {\tt cb\_lee@math.mit.edu}.
Research supported in part by NSF Grant DMS-1362326.}
\and
Benny Sudakov \thanks{Department of Mathematics, ETH, 8092 Zurich.
Email: {\tt benjamin.sudakov@math.ethz.ch}. 
Research supported in part by SNSF grant 200021-149111 and
by a USA-Israel BSF grant.}
}

\theoremstyle{plain}
\newtheorem{THM}{Theorem}[section]
\newtheorem{PROP}[THM]{Proposition}
\newtheorem{LEMMA}[THM]{Lemma}

\newtheorem{COR}[THM]{Corollary}
\newtheorem{CLAIM}[THM]{Claim}

\newtheorem{QUES}[THM]{Question}
\newtheorem{PROB}[THM]{Problem}

\theoremstyle{definition}

\renewenvironment{proof}
      {\medskip\noindent{\bf Proof.}\hspace{1mm}}
      {\hfill$\Box$\medskip}
\newenvironment{proofof}[1]
      {\medskip\noindent{\bf Proof of #1.}\hspace{1mm}}
      {\hfill$\Box$\medskip}

\begin{document}
\maketitle

\begin{abstract}
The Hales--Jewett theorem is one of the pillars of Ramsey theory, from which many other results follow.
A celebrated theorem of Shelah says that Hales--Jewett numbers are primitive recursive. A key tool used in his proof, now known as the cube lemma, has become famous in its own right. In its simplest form, this lemma says that if we color the edges of the Cartesian product $K_n \times K_n$ in $r$ colors then, for $n$ sufficiently large, there is a rectangle with both pairs of opposite edges receiving the same color. Shelah's proof shows that $n = r^{\binom{r+1}{2}} + 1$ suffices. More than twenty years ago, Graham, Rothschild and Spencer asked whether this bound can be improved to a polynomial in $r$. We show that this is not possible by providing a superpolynomial lower bound in $r$. We also discuss a number of related problems.
\end{abstract}

\section{Introduction}
\label{sec:intro}

Ramsey theory refers to a large body of deep results in mathematics whose underlying philosophy is captured succinctly by the statement that ``Every large system contains a large well-organized subsystem.''  This is an area in which a great variety of techniques from many branches of mathematics are used and whose results are important not only to combinatorics but also to logic, analysis, number theory and geometry. 

One of the pillars of Ramsey theory, from which many other results follow, is the Hales--Jewett theorem \cite{HJ63}. This theorem may be thought of as a statement about multidimensional multiplayer tic-tac-toe, saying that in a high enough dimension one of the players must win. However, this fails to capture the importance of the result, which easily implies van der Waerden's theorem on arithmetic progressions in colorings of the integers and its multidimensional generalizations. To quote \cite{GrRoSp}, ``The Hales--Jewett theorem strips van der Waerden's theorem of its unessential elements and reveals the heart of Ramsey theory. It provides a focal point from which many results can be derived and acts as a cornerstone for much of the more advanced work."

To state the Hales--Jewett theorem formally requires some notation. Let $[m]$ be the set of integers $\{1, 2, \dots, m\}$. We will refer to elements of $[m]^n$ as points or words and the set $[m]$ as the alphabet. For any $a \in [m]^n$, any $x \in [m]$ and any set $\gamma \subset [n]$, let $a\oplus x\gamma$ be the point of $[m]^n$ whose $j$-th component is $a_j$ if $j \not\in \gamma$ and $x$ if $j \in \gamma$. A combinatorial line is a subset of $[m]^n$ of the form $\{a \oplus x\gamma: 1 \leq x \leq m\}$. The Hales--Jewett theorem is now as follows.

\begin{THM}
For any $m$ and $r$ there exists an $n$ such that any $r$-coloring of the elements of $[m]^n$ contains a monochromatic combinatorial line.
\end{THM}

The original proof of the Hales--Jewett theorem is similar to that of van der Waerden's theorem, using a double induction, where we prove the theorem for alphabet $[m+1]$ and $r$ colors by using the statement for alphabet $[m]$ and a much larger number of colors. This results in bounds of Ackermann type for the dependence of the dimension $n$ on the size of the alphabet $m$. In the late eighties, Shelah \cite{Shelah} made a major breakthrough by finding a way to avoid the double induction and prove the theorem with primitive recursive bounds. This also resulted in the first primitive recursive bounds for van der Waerden's theorem (since drastically improved by Gowers \cite{G01}). 

Shelah's proof relied in a crucial way on a lemma now known as the Shelah cube lemma. The simplest case of this lemma says that if we color the edges of the Cartesian product $K_n \times K_n$ in $r$ colors then, for $n$ sufficiently large, there is a rectangle with both pairs of opposite edges receiving the same color. Shelah's proof shows that we may take $n \leq r^{\binom{r+1}{2}} + 1$. In the second edition of their book on Ramsey theory \cite{GrRoSp}, Graham, Rothschild and Spencer asked whether this bound can be improved to a polynomial in $r$. Such an improvement, if it could be generalized, would allow one to improve Shelah's wowzer-type upper bound for the Hales--Jewett theorem to a tower-type bound. The main result of this paper, Theorem~\ref{thm:grid_main} below, answers this question in the negative by providing a superpolynomial lower bound in $r$. We will now discuss this basic case of Shelah's cube lemma, which we refer to as the grid Ramsey problem, in more detail.

\subsection{The grid Ramsey problem}
\label{subsec:intro_grid}

For positive integers $m$ and $n$, let the \emph{grid graph} $\Gamma_{m,n}$
be the graph on vertex set $[m] \times [n]$ where two distinct 
vertices $(i,j)$ and $(i', j')$ are adjacent if and only if 
either $i=i'$ or $j=j'$. That is, $\Gamma_{m,n}$ is the Cartesian product $K_m \times K_n$.
A \emph{row} of the grid graph $\Gamma_{m,n}$ is a subgraph induced on a
vertex subset of the form $\{i\} \times [n]$ and a \emph{column}
is a subgraph induced on a
vertex subset of the form $[m] \times \{j\}$.

A rectangle in $\Gamma_{m,n}$ is a copy of $K_2 \times K_2$, that is, an induced
subgraph over a vertex subset of the form $\{(i,j), (i',j), (i,j'),(i',j') \}$ 
for some integers
$1 \le i < i' \le m$ and $1 \le j < j' \le n$. We will usually denote this rectangle
by $(i,j,i',j')$. For an edge-colored grid graph,
an \emph{alternating rectangle} is a rectangle $(i,j,i',j')$
such that the color of the edges $\{(i,j), (i',j)\}$ and
$\{(i,j'), (i',j')\}$ are equal and the color of the edges
$\{(i,j), (i,j')\}$ and $\{(i',j), (i',j')\}$ are equal, that is, opposite sides of the rectangle 
receive the same color.
An edge coloring of a grid graph is \emph{alternating-rectangle-free}
(or \emph{alternating-free}, for short) if
it contains no alternating rectangle. The function we will be interested in estimating is the following.

\medskip

\noindent \textbf{Definition.}
(i) For a positive integer $r$, define $G(r)$ as the minimum 
integer $n$ for which
every edge coloring of $\Gamma_{n,n}$ with $r$ colors contains an 
alternating rectangle.

\noindent (ii) For positive integers $m$ and $n$, define $g(m, n)$ as the
minimum integer $r$ for which there exists an alternating-free
edge coloring of $\Gamma_{m,n}$ with $r$ colors. Define $g(n) = g(n,n)$.

\medskip

Note that the two functions $G$ and $g$ defined above 
are in inverse relation to each other, in the sense that
$G(r) = n$ implies $g(n-1) \le r$ and $g(n) \ge r+1$,
while $g(n) = r$ implies $G(r) \ge n+1$ and $G(r-1) \le n$.

We have already mentioned Shelah's bound $G(r) \le r^{r+1 \choose 2}$ + 1.
To prove this, let $n = r^{r+1 \choose 2} + 1$ and suppose that an
$r$-coloring of $\Gamma_{r+1, n}$ is given.
There are at most $r^{r+1 \choose 2}$ ways that one can
color the edges of a fixed column with $r$ colors. Since the number of columns
is $n > r^{r+1 \choose 2}$, the pigeonhole principle implies that
there are two columns which are identically colored. Let
these columns be the $j$-th column and the $j'$-th column and consider the
edges that connect these two columns. Since there are
$r+1$ rows, the pigeonhole principle implies that there are two rows
that have the same color. Let these be the $i$-th row and the
$i'$-th row. Then the rectangle $(i,j, i', j')$ is alternating. 
Hence, we see that $G(r) \le n$, as claimed. This argument in fact
establishes the stronger bound
$g(r+1, r^{r+1 \choose 2} + 1) \ge r+1$.

It is surprisingly difficult to improve on this  
simple bound. The only known improvement, $G(r) \leq 
r^{\binom{r+1}{2}} - r^{\binom{r-1}{2} + 1} + 1$, which improves 
Shelah's bound by an additive term, was given by Gy\'arf\'as \cite{Gyarfas}.
Instead, we have focused on the lower bound, proving that $G(r)$ is 
superpolynomial in $r$. As already mentioned, this addresses a question
of Graham, Rothschild and Spencer \cite{GrRoSp}. This question was also 
reiterated by Heinrich \cite{H90} and by Faudree, Gy\'arf\'as and Sz\H{o}nyi \cite{FaGySz},
who proved a lower bound of $\Omega(r^3)$. 
Quantitatively, our main result is the following.

\begin{THM} \label{thm:grid_main}
There exists a positive constant $c$ such that
\[ G(r) > 2^{c (\log r)^{5/2}/\sqrt{\log \log r}}. \]
\end{THM}

We will build up to this theorem, first giving a substantially simpler proof for the weaker bound $G(r) > 2^{c \log^2 r}$. The following theorem, which includes this result, also contains a number of stronger bounds for the off-diagonal case, improving results of Faudree, Gy\'arf\'as and Sz\H{o}nyi \cite{FaGySz}.

\begin{THM} \label{thm:grid_asymmetric}
\begin{enumerate}
  \setlength{\itemsep}{1pt} \setlength{\parskip}{0pt}
  \setlength{\parsep}{0pt}
  \item[(i)] For all $C > e^2$, $g(r^{\log C/2}, r^{r/2C}) \le r$ for large enough $r$.
  \item[(ii)] For all positive constants $\varepsilon$, $g(2^{\varepsilon \log^2 r}, 2^{r^{1-\varepsilon}}) \le r$ for large enough $r$.
  \item[(iii)] There exists a positive constant $c$ such that $g(cr^2, r^{r^2 /2} / e^{r^2}) \le r$
  for large enough $r$.
\end{enumerate}
\end{THM}

Part (i) of this theorem already shows that $G(r)$ is superpolynomial in $r$, while part (ii) implies the more precise bound $G(r) > 2^{c \log^2 r}$ mentioned above, though at the cost of a weaker off-diagonal result. For $(m,n)=(r+1,r^{r+1 \choose 2})$, it is easy to find an alternating-free edge coloring of $\Gamma_{m,n}$ with $r$ colors by reverse engineering the proof of Shelah's bound $G(r) \leq r^{r+1 \choose 2} + 1$. Part (iii) of Theorem~\ref{thm:grid_asymmetric} shows that $n$ can be close to $r^{\binom{r+1}{2}}$ even when $m$ is quadratic in $r$. This goes some way towards explaining why it is so difficult to improve the bound $G(r) \le r^{r+1 \choose 2}$.

\subsection{The Erd\H{o}s--Gy\'arf\'as problem}
\label{subsec:intro_erdos_gyarfas}

The Ramsey-type problem for grid graphs 
considered in the previous subsection is closely related to a problem of Erd\H{o}s and
Gy\'arf\'as on generalized Ramsey numbers. To discuss this problem, we need some definitions.

\medskip 

\noindent \textbf{Definition.}
Let $k, p$ and $q$ be positive integers satisfying
$p \ge k+1$ and $2 \le q \le {p \choose k}$.
 
\noindent (i) For each positive integer $r$, define $F_k(r, p, q)$ as the 
minimum integer $n$ for which every edge coloring of $K^{(k)}_n$ with $r$ colors
contains a copy of $K^{(k)}_p$ receiving at most $q-1$ distinct colors
on its edges.

\noindent (ii) For each positive integer $n$, 
define $f_k(n,p,q)$ as the minimum integer $r$ for which 
there exists an edge coloring of $K^{(k)}_n$ with $r$ colors
such that every copy of $K^{(k)}_p$ receives at least $q$ distinct colors
on its edges.

\medskip

For simplicity, we usually write $F(r,p,q) = F_2(r,p,q)$ and $f(n,p,q) = f_2(n,p,q)$.
As in the previous subsection, the two functions are in 
inverse relation to each other: $F(r, p, q) = n$ implies $f(n, p, q) \ge r+1$
and $f(n-1, p, q) \le r$, 
while $f(n,p,q) = r$ implies $F(r, p, q)\ge n+1$
and $F(r-1,p,q) \le n$.

The function $F(r, p, q)$ generalizes the Ramsey number
since $F(2, p, 2)$ is equal to the Ramsey number $R(p)$.
Call an edge coloring of $K_n$ a
\emph{$(p,q)$-coloring} if every copy
of $K_p$ receives at least $q$ distinct colors on its edges. 
The definitions of $F(r,p,q)$ and $f(n,p,q)$ can also be reformulated
in terms of $(p,q)$-colorings. For example, $f(n, p, q)$ asks for the minimum
integer $r$ such that  there is a $(p,q)$-coloring of $K_n$ using $r$ colors.

The function $f(n,p,q)$ was first introduced by
Erd\H{o}s and Shelah \cite{E75, E81} and then was
systematically studied by Erd\H{o}s and Gy\'arf\'as \cite{ErGy}.
They studied the asymptotics of $f(n,p,q)$ 
as $n$ tends to infinity 
for various choices of parameters $p$ and $q$. It is clear that $f(n,p,q)$ is increasing 
in $q$, but it is interesting to understand the transitions in behavior as $q$ increases.
At one end of the 
spectrum, $f(n, p, 2) \le f(n,3,2) \le \lceil \log n \rceil$, while, at the other end, $f(n,p,{p \choose 2}) = {n \choose 2}$
(provided $p \ge 4$). In the middle range, Erd\H{o}s and Gy\'arf\'as proved that
$f(n,p,p) \ge n^{1/(p-2)} - 1$, which in turn implies that $f(n,p,q)$ is polynomial in $n$ for any $q \geq p$. 
A problem left open by Erd\H{o}s and Gy\'arf\'as was to determine whether $f(n, p, p-1)$ is also polynomial 
in $n$.

For $p = 3$, it is easy to see that $f(n, 3, 2)$ is subpolynomial since it is equivalent to determining
the multicolor Ramsey number of the triangle. For $p = 4$, Mubayi \cite{Mubayi} showed that 
$f(n,4,3) \leq 2^{c \sqrt{\log n}}$ and
Eichhorn and Mubayi \cite{EiMu} showed that
$f(n,5,4) \le 2^{c \sqrt{\log n}}$.
Recently, Conlon, Fox, Lee and Sudakov \cite{CoFoLeSu} 
resolved the question of Erd\H{o}s and Gy\'arf\'as, proving that
$f(n, p, p-1)$ is subpolynomial for all $p \geq 3$.
Nevertheless, the function $f(n,p,p-1)$ is still far
from being well understood, even for $p=4$, where the 
best known lower bound is $f(n,4,3) = \Omega(\log n)$ 
(see \cite{FoSu, KoMu}). 

In this paper, we consider extensions of these problems to hypergraphs.
The main motivation for studying this problem comes from an equivalent formulation
of the grid Ramsey problem (actually, this is Shelah's original formulation). 
Let $K^{(3)}(n,n)$ be the $3$-uniform hypergraph with vertex set
$A \cup B$, where $|A| = |B| = n$, and edge set consisting of all those triples which intersect both $A$ and $B$.
We claim that $g(n)$ is within a factor of two of the minimum integer
$r$ for which there exists an $r$-coloring of the edges of $K^{(3)}(n,n)$ such that 
any copy of $K_4^{(3)}$ has at least $3$ colors on its edges. 
To see the relation, we abuse notation and regard both $A$ and $B$ as copies of the set $[n]$.
For $i \in A$ and $j,j' \in B$, map the edge $\{(i,j), (i, j')\}$ of 
$\Gamma_{n,n}$ to the edge $(i,j,j')$ of $K^{(3)}(n,n)$ and,
for $i, i' \in A$ and $j \in B$, map the edge $\{(i,j), (i', j)\}$ of 
$\Gamma_{n,n}$ to the edge $(i,i',j)$ of $K^{(3)}(n,n)$.
Note that this defines a bijection between the edges of $\Gamma_{n,n}$
and the edges of $K^{(3)}(n,n)$, where the rectangles of $\Gamma_{n,n}$
are in one-to-one correspondence with 
the copies of $K_4^{(3)}$ of $K^{(3)}(n,n)$ intersecting both
sides in two vertices.
Hence, given a desired coloring of $K^{(3)}(n,n)$, we can find a 
coloring of $\Gamma_{n,n}$ where all rectangles receive at least
three colors (and are thus alternating-free), showing that $g(n) \le r$.
Similarly, given an alternating-free coloring of $\Gamma_{n,n}$, we may
double the number of colors to ensure that the set of colors used for row edges
and those used for column edges are disjoint. This turns an alternating-free
coloring of $\Gamma_{n,n}$ into a coloring where each $K_4^{(3)}$
receives at least three colors. Hence, as above, we see that $r \le 2 g(n)$.

Therefore, essentially the only difference between $g(n)$ and $f_3(2n, 4, 3)$ is that
the base hypergraph for $g(n)$ is $K^{(3)}(n,n)$ rather than $K_{2n}^{(3)}$. This observation allows
us to establish a close connection between the quantitative estimates for  
$f_3(n,4,3)$ and $g(n)$, as exhibited by
the following pair of inequalities (that we will prove in
Proposition \ref{prop:relation_g_f}):
\begin{eqnarray} \label{eq:relation_g_f}
g(n) \le f_3(2n, 4,3) \le 2 \lceil \log n \rceil^2 g(n).
\end{eqnarray}
This implies upper and lower bounds for $f_3(n,4,3)$ and $F_3(r,4,3)$
analogous to those we have established for $g(n)$ and $G(r)$.
More generally, we have the following recursive upper bound for $F_k(r, p, q)$.

\begin{THM} \label{thm:step_down}
For positive integers $r, k$, $p$ and $q$, all greater than $1$
and satisfying $r \ge k$, $p \geq k+1$ and $2 \leq q \leq \binom{p}{k}$,
\[ F_k\left( r, p, q \right) \le r^{{F_{k-1}(r,p-1,q) \choose k-1}}. \]
The above is true even for $q > {p-1 \choose k-1}$, where we trivially
have $F_{k-1}(r,p-1,q) = p-1$.
\end{THM}

By repeatedly applying Theorem \ref{thm:step_down}, we see 
that for each fixed $i$ with $0 \le i \leq k$ and large enough $p$, 
\[ F_k\left(r, p, {p-i \choose k-i} + 1\right) \le r^{r^{\iddots^{r^{c_{k,p}}} }}, \]
where the number of $r$'s in the tower is $i$.
For $0 < i < k$, it would be interesting to establish a lower bound on $F_k(r, p, {p-i \choose k-i})$ 
that is significantly larger than this upper bound on $F_k(r,p,{p-i \choose k-i} + 1)$.
This would establish an interesting phenomenon of `sudden jumps' 
in the asymptotics of $F_k(r,p,q)$ at the values $q = {p-i \choose k-i}$.
We believe that these jumps indeed occur. 

Let us examine some small cases of this problem. For graphs, as mentioned above,
$F(r,p,p)$ is polynomial while $F(r,p,p-1)$ is superpolynomial.
For 3-uniform hypergraphs, $F_3(r, p, {p-1 \choose 2} + 1)$ is polynomial
in terms of $r$. Hence, the first interesting case is to decide
whether the function $F_3(r, p, {p-1 \choose 2})$ is also
polynomial. The fact that $F_3(r,4,3)$ is superpolynomial
follows from Theorem~\ref{thm:grid_main}
and \eqref{eq:relation_g_f}, 
giving some evidence towards the conjecture that $F_3(r,p,{p-1 \choose 2})$
is superpolynomial. We provide further evidence by establishing
the next case for 3-uniform hypergraphs.

\begin{THM} \label{thm:F_3_5_6}
There exists a positive constant $c$ such that $F_3(r,5,6) \ge 2^{c\log^2 r}$
for all positive integers $r$.
\end{THM}

\subsection{A chromatic number version of the Erd\H{o}s--Gy\'arf\'as problem}

A graph with chromatic number equal to $p$ we call {\it $p$-chromatic}. In the process of studying $G(r)$ (and proving Theorem \ref{thm:grid_main}), we encountered
the following variant of the functions discussed
in the previous subsection, where $K_p$ is replaced by $p$-chromatic subgraphs.

\medskip 

\noindent \textbf{Definition.}
Let $p$ and $q$ be positive integers satisfying
$p \ge 3$ and $2 \le q \le {p \choose 2}$.
 
\noindent (i) For each positive integer $r$, define $F_\chi(r, p, q)$ as the 
minimum integer $n$ for which every edge coloring of $K_n$ with $r$ colors
contains a $p$-chromatic subgraph receiving 
at most $q-1$ distinct colors on its edges.

\noindent (ii) For each positive integer $n$, 
define $f_\chi(n,p,q)$ as the minimum integer $r$ for which 
there exists an edge coloring of $K_n$ with $r$ colors
such that every $p$-chromatic subgraph
receives at least $q$ distinct colors on its edges.

\medskip

Call an edge coloring of $K_n$ a \emph{chromatic-$(p,q)$-coloring} if every $p$-chromatic subgraph
receives at least $q$ distinct colors on its edges. As before, the
definitions of $F_{\chi}(r,p,q)$ and $f_{\chi}(n,p,q)$ can be restated in terms of
chromatic-$(p,q)$-colorings.
Also, an edge coloring of $K_n$ is a chromatic-$(p,q)$-coloring if and only
if the union of every $q-1$ color classes induces a graph
of chromatic number at most $p-1$. In some sense, 
this looks like the most natural interpretation for the functions
$F_\chi(r,p,q)$ and $f_\chi(n,p,q)$. If we choose to use this definition,
then it is more natural to shift the numbers $p$ and $q$ by $1$. However,
we use the definition above in order to make the connection 
between $F_\chi(r,p,q)$ and $F(r,p,q)$ more transparent. 

From the definition, we can immediately deduce some simple facts such as
\begin{align} \label{eq:intro_chi_1}
F_{\chi}(r,p,q) \le F(r,p,q), \qquad f_\chi(n,p,q) \ge f(n,p,q)
\end{align}
for all values of $r,p,q, n$ and
\[
f_{\chi}\left(n,p, {p \choose 2}\right) = f\left(n,p, {p \choose 2}\right)
= {n \choose 2}
\]
for all $n \ge p \ge 4$.

Let $n = F_{\chi}(r,p,q)-1$ and
consider a chromatic-$(p,q)$-coloring of $K_n$ that uses $r$ 
colors in total. Cover the
set of $r$ colors by $\lceil r/(q-1) \rceil$ subsets 
of size at most $q-1$. The chromatic number of the graph induced by
each subset is at most $p-1$ and thus, by the product formula
for chromatic number, we see that
\begin{align*} \label{eq:intro_chi_upper}
F_{\chi}(r, p,q) - 1 \le (p-1)^{\lceil r/(q-1) \rceil}.
\end{align*}
This gives a general exponential upper bound. 

On the other hand, when $d$ is a positive integer,
$p = 2^d + 1$ and $q=d+1$, a coloring
of the complete graph which we will describe in Section~\ref{sec:preliminaries}
implies that
\[ F_{\chi}(r, 2^d + 1, d + 1) \ge 2^{r} + 1. \]
Whenever $r$ is divisible by $d$, we see that the two bounds 
match to give 
\begin{align} \label{eq:intro_chi_2}
F_{\chi}(r, 2^d + 1, d + 1) = 2^{r} + 1.
\end{align}

Let us examine the value of $F_\chi(r,p,q)$ for some small values of $p$ and $q$.
By using the observations \eqref{eq:intro_chi_1} and \eqref{eq:intro_chi_2}, we have
\[
F_\chi(r,3,2) =  2^r + 1, \qquad F_\chi(r,3,3) \le F(r, 3, 3) \le r+2 \]
for $p=3$ and 
\[
F_\chi(r,4,2) \geq  2^r + 1, \qquad F_\chi(r,4,4) \le F(r, 4, 4) \le r^2 + 2
\]
for $p=4$ and $r \geq 2$. The bound on $F(r,4,4)$ follows since every edge coloring of 
the complete graph on $r^2+2$ vertices with $r$ colors contains a vertex $v$ and a set $X$ of size at least
$r+1$ such that all the edges connecting $v$ to $X$ have the same color, say red. If an edge $e$ with vertices in $X$ is red, we get a $K_4$ using at most three colors by taking $v$, the vertices of $e$, and any other vertex in $X$. So we may suppose at most $r-1$ colors are used on the edges with both vertices in $X$. In this case, as $|X| \ge F(r-1,3,3)$, we can find three vertices in $X$ with at most 
two colors used on the edges connecting them. 
Adding $v$ to this set gives a set of four vertices with
at most three colors used on the edges connecting them.
 
We show that the asymptotic behavior of $F_\chi(r,4,3)$ 
is different from both $F_\chi(r,4,2)$ and $F_\chi(r,4,4)$.

\begin{THM} \label{thm:chi_4_3}
There exist positive constants $C$ and $r_0$ such that for all $r \ge r_0$,
\[ 2^{\log^2 r/36} \le F_{\chi}(r, 4, 3) \le C \cdot 2^{130 \sqrt{r\log r}}. \] 
\end{THM}

Despite being interesting in its own right, our principal motivation for considering chromatic-$(p,q)$-colorings was to establish the following theorem,
which is an important ingredient in the proof of Theorem \ref{thm:grid_main}.

\begin{THM} \label{thm:chi_slow_grow}
For every fixed positive integer $n$, 
there exists an edge coloring of the complete graph $K_n$ with $2^{6\sqrt{\log n}}$ colors 
with the following property: for every
subset $X$ of colors with $|X| \ge 2$, the subgraph induced by the edges colored
with a color from $X$ has chromatic number at most $2^{3 \sqrt{|X| \log |X|}}$.
\end{THM}

This theorem has the following immediate corollary.

\begin{COR}
For all positive integers $n, r$ and $s \geq 2$, 
\[ 
	f_{\chi}(n, 2^{3\sqrt{s \log s}} + 1, s + 1) \le 2^{6\sqrt{\log n}}
	\quad \text{and} \quad
	F_{\chi}(r, 2^{3\sqrt{s \log s}} + 1, s + 1) \ge 2^{\log^2 r/36}.
\]
\end{COR}

Our paper is organized as follows.
In Section \ref{sec:preliminaries}, we review two coloring functions that 
will be used throughout the paper. 
In Section \ref{sec:gridramsey}, we prove Theorems \ref{thm:grid_main} and \ref{thm:grid_asymmetric}. 
In Section \ref{sec:eg}, we prove Theorems \ref{thm:step_down} and \ref{thm:F_3_5_6}.
In Section \ref{sec:chi_eg}, we prove Theorems \ref{thm:chi_4_3} and \ref{thm:chi_slow_grow}.
We conclude with some further remarks and open problems in Section~\ref{sec:conclusion}.

\medskip

\noindent \textbf{Notation.}
We use $\log$ for the base 2 logarithm
and $\ln$ for the natural logarithm. For the sake of 
clarity of presentation, we systematically omit floor and ceiling signs 
whenever they are not essential.  We also do not make any serious attempt 
to optimize absolute constants in our statements and proofs. 

The following standard asymptotic notation will be used throughout. 
For two functions $f(n)$ and $g(n)$, we write $f(n) =
o(g(n))$ if $\lim_{n \rightarrow \infty} f(n)/g(n) = 0$ and $f(n) =
O(g(n))$ or $g(n) = \Omega(f(n))$ if there exists a constant $M$ such
that $|f(n)| \leq M|g(n)|$ for all sufficiently large $n$.  We also
write $f(n) = \Theta(g(n))$ if both $f(n) = O(g(n))$ and $f(n) =
\Omega(g(n))$ are satisfied.

\section{Preliminaries}
\label{sec:preliminaries}

We begin by defining two edge colorings of the complete graph $K_n$, one a $(3,2)$-coloring
and the other a $(4,3)$-coloring. These will be used throughout the paper.

We denote the $(3,2)$-coloring by $c_B$, where `B' stands
for `binary'. To define this coloring, we let $t$ be the smallest integer such that
$n \leq 2^t$. We consider the vertex set $[n]$ as a subset of $\{0,1\}^t$ 
by identifying $x$ with $(x_1, \dots, x_t)$, where $\sum_{i=1}^t x_i 2^{i-1}$ is the 
binary expansion of $x - 1$. Then, for two vertices $x = (x_1, \dots, x_t)$ and $y = (y_1, \dots, y_t)$,
$c_B(x,y)$ is the minimum $i$ for which $x_i \neq y_i$.
This coloring uses at most $\lceil \log n \rceil$ colors and
is a $(3,2)$-coloring since three vertices cannot all differ in the $i$-th coordinate.
In fact, it is a chromatic-$(2^d + 1, d+1)$-coloring for all integers $d \ge 1$,
since it gives an edge partition
$E = E_1 \cup \dots \cup E_{t}$ of $K_n$ for $t = \lceil \log n \rceil$ such that, 
for all $J \subset [t]$, the graph consisting of the edges $\bigcup_{j \in J} E_j$
has chromatic number at most $2^{|J|}$.

The $(4,3)$-coloring, which is a 
variant of Mubayi's coloring \cite{Mubayi}, will be denoted by $c_M$. To define this coloring,
we let $t$ be the smallest integer such that $n \leq 2^{t^2}$ and $m = 2^t$. We consider the vertex
set $[n]$ as a subset of $[m]^t$ by identifying $x$ with $(x_1, \dots, x_t)$, this time by examining the
base $m$ expansion of $x-1$. For two vertices
$x = (x_1, \ldots,x_t)$ and $y = (y_1 ,\ldots, y_t)$, let
\[ c_M(x,y) = \Big( \{x_i, y_i\}, a_1, \ldots,a_t \Big), \] 
where $i$ is the minimum index in which $x$ and $y$ differ
and $a_j = 0$ or $1$ depending on whether $x_j = y_j$ or not (note that the coloring function 
$c_M$ is symmetric in its two variables).

The coloring $c_M$ is a $(3,2)$-coloring since it partitions the edge set of $K_n$ into bipartite graphs 
and a simple case analysis similar to that given in \cite{Mubayi} shows that it is a $(4,3)$-coloring (the proof of Theorem \ref{thm:chi_4_3} given in Section \ref{sec:chi_eg} shows that $c_M$ is even a chromatic-$(4,3)$-coloring).
Since $2^{(t-1)^2} < n$, the total number of colors used is at most
\[ m^2 \cdot 2^t = 2^{3 t} < 2^{3 (1 + \sqrt{\log n})} \leq 2^{6 \sqrt{\log n}}. \]
Hence, $c_M$ uses at most $r = 2^{6\sqrt{\log n}}$ colors
to color the edge set of the complete graph on 
$n = 2^{\log^2 r/ 36}$ vertices.

\section{The grid Ramsey problem}
\label{sec:gridramsey}

In order to improve the lower bound on $G(r)$, we need to find
an edge coloring of the grid graph which is alternating-free.
The following lemma is the key idea behind our argument. 
For two edge-coloring functions $c_1$ and $c_2$ of the complete
graph $K_n$, let $\mathcal{G}(c_1, c_2)$ be the subgraph
of $K_n$ where $e$ is an edge if and only 
if $c_1(e) = c_2(e)$.

\begin{LEMMA} \label{lem:row_chromatic}
Let $m, n$ and $r$ be positive integers.
There exists an alternating-rectangle-free edge coloring 
of $\Gamma_{m,n}$ with $r$ colors if and only if
there are edge colorings $c_1, \ldots, c_m$ of 
the complete graph $K_n$ with $r$ colors satisfying
\[ \chi(\mathcal{G}(c_i, c_j)) \le r \]
for all pairs of indices $i,j$.
\end{LEMMA}
\begin{proof}
We first prove the `if' statement.
Consider the grid graph $\Gamma_{m,n}$. For each $i$, color the edges of
the $i$-th row using the edge coloring $c_i$. 
Then, for each distinct pair of indices $i$ and $i'$, 
construct auxiliary graphs $H_{i,i'}$ whose vertex 
set is the set of edges that connect
the $i$-th row with the $i'$-th row (that is, edges of the form $\{(i,j), (i',j)\}$) and
where two vertices 
$\{(i,j), (i',j)\}$ and  $\{(i,j'), (i',j')\}$
are adjacent if and only if the two row edges that 
connect these two column edges
have the same color. 

The fact that $\chi(\mathcal{G}(c_i, c_{i'})) \le r$ implies that
there exists a vertex coloring of $H_{i,i'}$ with $r$ colors. Color the 
corresponding edges $\{(i,j), (i',j)\}$ according to this vertex coloring. 
Under this coloring, we see that whenever
a pair of edges of the form $\{(i,j), (i,j')\}$ and $\{(i',j), (i',j')\}$
have the same color, the colors of the edges 
$\{(i,j), (i',j)\}$ and $\{(i,j'), (i',j')\}$ are distinct.
This gives a coloring of the column edges. 
Hence, we found the required alternating-rectangle-free edge
coloring of $\Gamma_{m,n}$ with $r$ colors.

For the `only if' statement, given an alternating-rectangle-free
edge coloring of $\Gamma_{m,n}$ with $r$ colors, define
$c_i$ as the edge-coloring function of the $i$-th row of $\Gamma_{m,n}$,
for each $i \in [m]$. One can easily reverse the 
process above to show that the colorings $c_1, \ldots, c_m$ satisfy 
the condition. We omit the details.
\end{proof}

To find an alternating-rectangle-free 
edge coloring of $\Gamma_{m,n}$, we will find
edge colorings of the rows which satisfy the condition 
of Lemma \ref{lem:row_chromatic}. 
Suppose that $E(K_n) = E_1 \cup \dots \cup E_t$ is a partition
of the edge set of $K_n$. For an index subset $I \subset [t]$, 
we let $\mathcal{G}_I$ be the subgraph of $K_n$ whose edge set is
given by $\bigcup_{i\in I} E_i$.

\begin{LEMMA} \label{lem:grid_coloring}
Let $m, n, r$ and $t$ be positive integers. Suppose that an
edge partition $E(K_n) = E_1 \cup \dots \cup E_t$ of $K_n$ is given. 
Let $I$ be a random subset of $[t]$ where 
each element in $I$ is chosen independently with probability $1/r$ and suppose that
\[ \mathbb{P}[\chi(\mathcal{G}_I) \ge r+1] \le \frac{1}{2m}. \]
Then $g(m, n) \le r$ and $G(r) \ge \min\{m,n\} + 1$.
\end{LEMMA}

\begin{proof}
For each $i \in [2m]$, choose a vector $v_i \in [r]^t$ 
independently and uniformly at random.
Let $c_i$ be an edge coloring of $K_n$ with $r$ colors where for 
each $t' \in [t]$, we color all the edges in $E_{t'}$ using the
value of the $t'$-th coordinate of $v_i$.
Color the $i$-th row of $\Gamma_{2m,n}$ (which is a copy of $K_n$) 
using $c_i$.

For a pair of distinct indices $i,j \in [2m]$, let $I(i,j)$ 
be the subset of indices $t' \in [t]$ for which $v_i$ and $v_j$
have the same value on their $t'$-th coordinates (thus implying that
$c_i$ and $c_j$ use the same color on $E_{t'}$). Then
$I(i,j)$ has the same distribution as a random subset of 
$[t]$ obtained by taking each element independently with probability
$1/r$. Moreover,
\[ \mathcal{G}(c_i, c_j) = \mathcal{G}_{I(i,j)}. \]
Hence, 
\[ \mathbb{P}[\chi(\mathcal{G}(c_i, c_j)) \ge r+1] = \mathbb{P}[\chi(\mathcal{G}_{I(i,j)}) \ge r+1] \le \frac{1}{2m}. \]

Therefore, the expected number of pairs $i,j$ with $i<j$ having
$\chi(\mathcal{G}_{I(i,j)}) \ge r+1$ is at most ${2m \choose 2} \frac{1}{2m} \le m$. 
Hence, there exists a choice of coloring functions $c_i$ for which this number is 
at most $m$. If this event happens, then we can remove one row
from each pair $i,j$ having $\chi(\mathcal{G}_{I(i,j)}) \ge r+1$
to obtain a set $R \subset [2m]$ of size 
at least $m$ which has the property that $\chi(\mathcal{G}_{I(i,j)}) \le r$
for all $i,j \in R$. By considering the subgraph of $\Gamma_{2m,n}$
induced on $R \times [n]$ and using Lemma \ref{lem:row_chromatic}, we obtain
an alternating-rectangle-free edge coloring of $\Gamma_{m,n}$ with $r$ colors.
The result follows.
\end{proof}

We prove Theorems \ref{thm:grid_main} and \ref{thm:grid_asymmetric}
in the next two subsections. 
We begin with Theorem \ref{thm:grid_asymmetric}, which establishes upper bounds for 
$g(m,n)$ in various off-diagonal regimes. As noted in the introduction, parts (i) and (ii) already yield weak versions of 
Theorem~\ref{thm:grid_main}. In particular, part (i) implies that $G(r)$ is superpolynomial in $r$, 
while part (ii) yields the bound $G(r) > 2^{c \log^2 r}$. We recall the stronger off-diagonal statements below.

\subsection{Proof of Theorem \ref{thm:grid_asymmetric}}
\label{subsec:grid_asymmetric}
  
\noindent \textbf{Parts (i) and (ii)} : For all $C > e^2$, $\varepsilon > 0$
and large enough $r$, $g(r^{\log C/2}, r^{r/2C}) \le r$ and 
$g(2^{\varepsilon \log^2 r}, 2^{r^{1-\varepsilon}}) \le r$.

\medskip

Let $n = 2^t$ for some $t$ to be chosen later. The edge coloring $c_B$
from Section \ref{sec:preliminaries} gives an edge partition
$E = E_1 \cup \dots \cup E_{t}$ of $K_n$ for $t = \log n$ such that, 
for all $J \subset [t]$, 
\[ \chi(\mathcal{G}_J) = 2^{|J|}. \]
Hence, if we let $I$ be a random subset of $[t]$ obtained by choosing each element
independently with probability $1/r$, then  
\begin{align}
\mathbb{P}[\chi(\mathcal{G}_I) \ge r+1] 
&= \mathbb{P}\big[|I| \ge \log(r+1)\big] \nonumber \\
&\le {t \choose \log(r+1)} \frac{1}{r^{\log(r+1)}}
 \le \left( \frac{et}{r \log (r+1)} \right)^{\log (r+1)}. \label{eq:grid_result1}
\end{align}

For part (i), let $C$ be a given constant and take $t = r\log r /2C$. 
Then the right-hand side of \eqref{eq:grid_result1}
is at most $(r+1)^{-\log(C/e)}$. In Lemma \ref{lem:grid_coloring},
we can take $m = \frac{1}{2}(r+1)^{\log (C/e)} \ge r^{\log C/2}$ 
and $n = 2^{t} = r^{r/2C}$ to get
\[ g(r^{\log C/2}, r^{r/2C}) \le r.  \]

For part (ii), let $\varepsilon$ be a given constant
and take $t = r^{1-\varepsilon}$. For large enough $r$,
the right-hand side of \eqref{eq:grid_result1} is at 
most $\frac{1}{2}r^{-\varepsilon \log r}$. Hence,
by applying Lemma \ref{lem:grid_coloring} with $m = r^{\varepsilon \log r} = 
2^{\varepsilon \log^2 r}$ and $n = 2^t = 2^{r^{1-\varepsilon}}$, we see that
\[ g(2^{\varepsilon \log^2 r}, 2^{r^{1-\varepsilon}}) \le r. \]

\medskip

\noindent \textbf{Part (iii)} : There exists a positive constant $c$ such that $g(cr^2, r^{r^2 /2} / e^{r^2}) \le r$ for large enough $r$.

\medskip

Let $c=e^{-3}$. Let $n = cr^2$ and partition the edge set of $K_n$ into $t = {n \choose 2}$
sets $E_1, \ldots, E_t$, each of size exactly one. As before, 
let $I$ be a random subset of $[t]$ obtained by choosing each element
independently with probability $1/r$.
In this case, we get $\mathcal{G}_I = \mathcal{G}(n, \frac{1}{r})$ 
(where $\mathcal{G}(n, p)$ is the binomial random graph). 
Therefore,
\[ \mathbb{P}[\chi(\mathcal{G}_I) \ge r+1] = \mathbb{P}[\chi (\mathcal{G}(cr^2, 1/r)) \ge r+1].  \]
The event $\chi\left(\mathcal{G}(cr^2, \frac{1}{r})\right) \ge r+1$ 
is contained in the event that
$\mathcal{G}(cr^2, \frac{1}{r})$ contains a subgraph of order $s \ge r+1$ 
of minimum degree at least $r$. 
The latter event has probability at most
\begin{align}
\label{eqn:g_n_p_chromatic}
 \sum_{s=r+1}^{cr^2} {cr^2 \choose s} {s^2/2 \choose rs/2} \left(\frac{1}{r}\right)^{rs/2}
 &\le 
 \sum_{s=r+1}^{cr^2} \left( \left( \frac{ecr^2}{s} \right)^{2}  \left(\frac{es}{r} \right)^{r} \left(\frac{1}{r} \right)^{r}  \right)^{s/2}.
\end{align}
For $s=r$, if $r$ is large enough, then the summand is
\[ \left( (ecr)^{2}  e^{r} \left(\frac{1}{r} \right)^{r}  \right)^{r/2}
   \le \frac{e^{r^2}}{r^{r^2 /2}}\,. \]
We next show that the summands are each at most a quarter of the previous summand. As the series starts at $s=r+1$ and ends at $s=cr^2$, the series is then at most half the summand for $s=r$. The ratio of the summand for $s+1$ to the summand for $s$, where $r+1 \leq s+1 \leq cr^2$, is 
$$
\left(\frac{s+1}{s}\right)^{s(r-2)/2} \left(\left(\frac{ecr^2}{s+1}\right)^2\left(\frac{e(s+1)}{r}\right)^r\left(\frac{1}{r}\right)^r\right)^{1/2}$$ 
which is at most 
$$e^{(r-2)/2} \left(\frac{e^{r+2} c^2 (s+1)^{r-2}}{r^{2r-4}}\right)^{1/2} \leq e^{(r-2)/2} (e^{r+2} c^r )^{1/2} = e^{-r/2} < \frac{1}{4},
$$
for $r$ sufficiently large. 

Hence, the right-hand side of \eqref{eqn:g_n_p_chromatic} is
at most $e^{r^2}r^{-r^2/2}/2$ and
\[ \mathbb{P}[\chi(\mathcal{G}_I) \ge r+1] \le \frac{e^{r^2}}{2r^{r^2 /2}}.  \]
By Lemma \ref{lem:grid_coloring}, we conclude that
$g(cr^2, r^{r^2 / 2} / e^{r^2}) \le r$.

\subsection{Proof of Theorem \ref{thm:grid_main}}
\label{subsec:grid_main}

In the previous subsection we used quite simple edge partitions of the complete graph
as an input to Lemma \ref{lem:grid_coloring} to prove Theorem \ref{thm:grid_asymmetric}.
These partitions were already good enough to give the superpolynomial bound $G(r) > 2^{c \log^2 r}$.
To further improve this bound and prove Theorem \ref{thm:grid_main}, 
we make use of a slightly more sophisticated edge 
partition guaranteed by the following theorem.

\begin{THM} \label{thm:chi_slow_grow_less_colors}
There exists a positive real $r_0$ such that the following holds
for positive integers $r$ and positive reals $\alpha \le 1$
satisfying $(\log r)^\alpha \ge r_0$.
For $n = 2^{(\log r)^{2 + \alpha}/200}$,
there exists a partition $E = E_1 \cup \dots \cup E_{\sqrt{r}}$ of
the edge set of the complete graph $K_n$ such that
\[ \chi(\mathcal{G}_I) \le 2^{3(\log r)^{\alpha/2}\sqrt{|I| \log 2|I|}} \]
for all $I \subset [\sqrt{r}]$.
\end{THM}

The proof of this theorem is based on 
Theorem \ref{thm:chi_slow_grow}, which is in turn based on considering the coloring $c_M$, and will be given in Section \ref{sec:chi_eg}.

Now suppose that a positive integer $r$ is given and let 
$\alpha \leq 1$ be a real to be chosen later. Let $E_1 \cup \dots \cup E_{\sqrt{r}}$
be the edge partition of $K_n$ for $n = 2^{(\log r)^{2 + \alpha}/200}$
given by Theorem \ref{thm:chi_slow_grow_less_colors}.
Let $I$ be a random subset of $[\sqrt{r}]$ chosen by taking each
element independently with probability $\frac{1}{r}$. Then,
by Theorem \ref{thm:chi_slow_grow_less_colors}, we have
\[ \chi(\mathcal{G}_I) \ge r+1 \quad \Rightarrow \quad |I| \ge c\frac{(\log r)^{2 - \alpha}}{\log \log r}, \]
for some positive constant $c$. Therefore,
\begin{align*}
\mathbb{P}[\chi(\mathcal{G}_I) \ge r+1  ]
&\le \mathbb{P}\left[|I| \ge c \frac{(\log r)^{2 - \alpha}}{\log \log r}\right] \\
&\le {\sqrt{r} \choose c(\log r)^{2 - \alpha}/\log \log r} \left(\frac{1}{r} \right)^{c(\log r)^{2 - \alpha}/ \log \log r} 
 \le r^{-c'(\log r)^{2 - \alpha}/\log \log r}
\end{align*}
holds for some positive constant $c'$. 
By Lemma \ref{lem:grid_coloring}, for $m = 2^{c'(\log r)^{3-\alpha} / \log \log r - 1}$, 
we have $g(m, n) \le r$. We may choose $\alpha$ so that 
\[ m = n = e^{\Omega((\log r)^{5/2}/ \sqrt{\log \log r})}. \]
This gives $G(r) \ge 2^{\Omega((\log r)^{5/2}/ \sqrt{\log \log r})}$, as required.

\section{The Erd\H{o}s--Gy\'arf\'as problem}
\label{sec:eg}

In the introduction, we discussed how the grid Ramsey problem is connected to a 
hypergraph version of the Erd\H{o}s--Gy\'arf\'as problem. We now establish this
correspondence more formally.

\begin{PROP} \label{prop:relation_g_f}
For all positive integers $n$, we have
\[ g(n) \le f_3(2n, 4,3) \le 2 \lceil \log n \rceil ^2 g(n). \]
\end{PROP}

\begin{proof}
Since $K^{(3)}(n,n)$ is a subhypergraph of $K_{2n}^{(3)}$,
a $(4,3)$-coloring of $K_{2n}^{(3)}$ immediately gives
a coloring of $K^{(3)}(n,n)$ such that every copy of $K_{4}^{(3)}$ receives at least three colors.
Hence, by the correspondence between coloring functions for $K^{(3)}(n,n)$ and
coloring functions for $\Gamma_{n,n}$ explained in the introduction, it follows that
$g(n) \le f_3(2n,4,3)$.

We prove the other inequality by showing that for all $m \le n$,
\begin{align}
\label{eq:recursive}
f_3(2m, 4, 3) \le f_3(m, 4, 3) + 2\lceil \log m \rceil g(m).
\end{align}
By repeatedly applying this recursive formula, we obtain the
claimed inequality
\begin{align*}
 f_3(2n, 4, 3)  \le 2\lceil \log n \rceil^2 g(n).
\end{align*}
Thus it suffices to establish the recursive formula \eqref{eq:recursive}.
We will do this by presenting a $(4,3)$-coloring
of $K^{(3)}_{2m}$. Let $A$ and $B$ be two disjoint vertex subsets of 
$K^{(3)}_{2m}$, each of order $m$. Given a 
$(4,3)$-coloring of $K_{m}^{(3)}$ with $f_3(m,4,3)$ colors,
color the hyperedges within $A$ using this coloring
and also the hyperedges within $B$ using this coloring. Since
we started with a $(4,3)$-coloring, every 
copy of $K_4^{(3)}$ lying inside $A$ or $B$ contains
at least $3$ colors on its edges. This
leaves us with the copies which intersect both
$A$ and $B$.

Let $H$ be the bipartite hypergraph that consists of the edges which
intersect both parts $A$ and $B$.
By definition, we have an alternating-free edge coloring
of the grid graph $\Gamma_{m,m}$ using $g(m)$ colors.
We may assume, by introducing at most $g(m)$ new colors, that the 
set of colors used for the row edges and the column edges are disjoint.
This gives an edge coloring of $\Gamma_{m,m}$, where each rectangle
receives at least three colors. 
Let $c_1$ be a coloring of $H$ using at most $2g(m)$ colors,
where for an edge $\{i,j,j'\} \in H$ with $i \in A, j,j' \in B$, 
we color it with the color of the edge $\{(i,j), (i,j')\}$ in $\Gamma_{m,m}$
and for an edge $\{i,i',j\} \in H$ with $i,i' \in A, j \in B$,
we color it with the color of the edge $\{(i,j), (i',j)\}$ in $\Gamma_{m,m}$.
Let $c_2$ be a coloring of $H$ constructed based on the
coloring $c_B$ given in Section \ref{sec:preliminaries} as follows:
for an edge $\{i,j,j'\} \in H$ with $i \in A, j,j' \in B$, 
let $c_2(\{i,j,j'\}) = c_B(\{j,j'\})$ and for an edge
$\{i,i',j\} \in H$ with $i,i' \in A, j \in B$,
let $c_2(\{i,i',j\}) = c_B(\{i,i'\})$. Now color the hypergraph
$H$ using the coloring function $c_1 \times c_2$. 

Consider a copy $K$ of $K_4^{(3)}$ which intersects both parts $A$ and $B$.
If $|K \cap A| = |K \cap B| = 2$, then assume that
$K = \{i,i',j,j'\}$ for $i,i' \in A$ and $j,j' \in B$. One can see
that the set of colors used by $c_1$ on $K$ is identical to the
set of colors used on the rectangle $(i,j,i',j')$ in $\Gamma_{m,m}$ 
considered above. Thus $K$ receives at least three distinct colors.
If $|K \cap A| = 1$ and $|K \cap B| = 3$, then the three hyperedges
in $K$ which intersect $A$ use at least two colors from the coloring $c_2$,
while the unique hyperedge of $K \cap B$ is colored with a different color.
Hence $K$ contains at least three colors. Similarly, $K$ contains 
at least three colors if $|K \cap A| = 3$ and $|K \cap B| = 1$.

Since $c_1$ uses at most $2g(m)$ colors and $c_2$ uses at most $\lceil \log m \rceil$ colors, we see that
$c_1 \times c_2$ uses at most $2 \lceil \log m \rceil g(m)$ colors. 
Recall that we used at most $f_3(m,4,3)$ colors to color
the edges inside $A$ and $B$. Therefore, we have found a $(4,3)$-coloring
of $K^{(3)}_{2m}$ using at most
\[ f_3(m,4,3) + 2 \lceil \log m \rceil g(m) \]
colors, thereby establishing \eqref{eq:recursive}.
\end{proof}

\subsection{A basic bound on $F_k(r,p,q)$}

Here we prove Theorem \ref{thm:step_down} that provides a basic upper bound on the function $F_k(r,p,q)$. Recall that we are
given positive integers $r, k, p$, and $q$ all greater than $1$ and
satisfying $r \ge k$.

Let $N = r^{{F_{k-1}(r,p-1,q) \choose k-1}}$ and suppose that we are given
an edge coloring of $K_N^{(k)}$ with $r$ colors (denoted by $c$). Let $[N]$ be the vertex set of $K_N^{(k)}$.
For each integer $t$ in the range $1 \le t \le F_{k-1}(r,p-1,q)$,
we will inductively find a pair of disjoint subsets $X_t$ and $Y_t$ of $[N]$
with the following properties:

\begin{enumerate}
\item $|X_t| = t$ and $|Y_t| \ge \min\{N / r^{{t \choose k-1}}, N-t\}$,
\item for all $x \in X_t$ and $y \in Y_t$, $x < y$,
\item for all edges $e \in {X_t \cup Y_t \choose k}$ satisfying $|e \cap X_t| \ge k-1$,
the color of $e$ is determined by the first $k-1$ elements of $e$ (note that the 
first $k-1$ elements necessarily belong to $X_t$). 
\end{enumerate}

For the base cases $t=1, \ldots, k-2$, the pair of sets $X_{t} = \{1,2,\ldots,t\}$
and $Y_{t} = [N] \setminus X_t$ trivially satisfy the given properties.
Now suppose that for some $t \ge k-2$, we are given pairs $X_{t}$ and $Y_{t}$ and 
wish to construct sets $X_{t+1}$ and $Y_{t+1}$.
Since $t < F_{k-1}(r,p-1,q)$, Property 1 implies that $|Y_t| \ge 1$ and in 
particular that $Y_t$ is nonempty.
Let $x$ be the minimum element of $Y_{t}$ and let $X_{t+1} = X_{t} \cup \{x\}$.

For each element $y \in Y_{t} \setminus \{x\}$, consider the 
vector of colors of length ${|X_{t}| \choose k-2}$ whose coordinates 
are $c(e' \cup \{x, y\})$ for each $e' \in {X_{t} \choose k-2}$.
By the pigeonhole principle, there are at least 
$\frac{|Y_t| - 1}{ r^{{ |X_t| \choose k-2 }} }$ vertices which
have the same vector. Let $Y_{t+1}$ be these vertices.
This choice immediately implies Properties 2 and 3 above.
To check Property 1, note that  
\[ 
	|Y_{t+1}| 
	\ge \frac{|Y_t| - 1}{ r^{{ |X_t| \choose k-2 }} }
	\ge \frac{N / r^{{t \choose k-1}} - t - 1}{ r^{{ |X_t| \choose k-2 }} }
	= \frac{N}{ r^{{ t+1 \choose k-1 }} } - \frac{t+1}{r^{{ t \choose k-2 }}}
	> \frac{N}{ r^{{ t+1 \choose k-1 }} } - 1,
\]
where the final inequality follows from $t \ge k-2$ and $r \ge k$.
Since $N = r^{{F_{k-1}(r,p-1,q) \choose k-1}}$, $F_{k-1}(r,p-1,q) \ge t+1$ 
and $|Y_{t+1}|$ is an
integer, this implies that $|Y_{t+1}| \ge \frac{N}{ r^{{ t+1 \choose k-1 }} }$.

Let $T = F_{k-1}(r,p-1,q)$ and note that $|X_T| = F_{k-1}(r,p-1,q)$
and $|Y_T| \ge 1$. Construct an auxiliary complete $(k-1)$-uniform
hypergraph over the vertex set $X_T$ and color each edge 
with the color guaranteed by Property 3 above. This gives an edge coloring
of $K^{(k-1)}_T$ with $r$ colors and thus, by definition, we can find a
set $A$ of $p-1$ vertices using fewer than $q$ colors
on its edges in the auxiliary $(k-1)$-uniform hypergraph. 
It follows from Property 3 that for an arbitrary $y \in Y_T$, 
$A \cup \{y\}$ is a set of $p$ vertices using fewer than $q$ colors
on its edges in the original $k$-uniform hypergraph.

\subsection{A superpolynomial lower bound for $F_3(r, 5, 6)$}

In this subsection, we present a $(5,6)$-coloring of 
$K_n^{(3)}$ using $2^{O(\sqrt{\log n})}$ colors.
This shows that $f_3(n,5,6) = 2^{O(\sqrt{\log n})}$ 
and $F_3(r,5,6) = 2^{\Omega(\log^2 r)}$.

The edge coloring is given as a product $c = c_1 \times c_2 \times c_3 \times c_4$ of
four coloring functions $c_1,c_2,c_3,c_4$. The first coloring $c_1$ is a 
$(4,3)$-coloring of $K_n^{(3)}$ using $f_3(n,4,3)$ colors. Combining
Proposition \ref{prop:relation_g_f} and Theorem \ref{thm:grid_main}, we see that
$f_3(n,4,3) = 2^{O((\log n)^{2/5} (\log \log n)^{1/5})}$.

Let $n=2^d$ and write the vertices of $K_n$ as binary strings of length $d$.
To define $c_2, c_3$ and $c_4$, 
for three distinct vertices $u,v,w$, assume that the least coordinate
in which not all vertices have the same bit is the $i$-th coordinate
and let $u_i, v_i, w_i$ be the $i$-th coordinate of $u,v,w$, respectively.
Without loss of generality, we may assume that $u_i = v_i \neq w_i$, 
i.e. $(u_i, v_i, w_i) = (0,0,1)$ or $(1,1,0)$.
Define the second color $c_2$ of the triple of vertices $\{u,v,w\}$ as $i$.
Thus $c_2$ uses at most $\log n$ colors. 
Define the third color $c_3$ as the value of $w_i$, which is either 0 or 1.
Define the fourth color $c_4$ as $c_M(u, v)$, where 
$c_M$ is the graph coloring
given in Section \ref{sec:preliminaries}, which is both
a $(3,2)$ and $(4,3)$-coloring. Recall that 
$c_M$ uses at most $2^{O(\sqrt{\log n})}$ colors.

The number of colors in the coloring $c$ is  
\[ 2^{O((\log n)^{2/5} (\log \log n)^{1/5})} \cdot \log n 
\cdot 2 \cdot 2^{O(\sqrt{\log n})}  =  2^{O(\sqrt{\log n})}, \]
as desired. Now we show that each set of $5$ vertices receives at 
least $6$ colors in the coloring $c$. Let $i$ be the least 
coordinate such that the five vertices do not all agree.

\medskip

\noindent \textbf{Case 1}: One of the vertices (call it $v_1$) has one bit at coordinate $i$, while
the other four vertices (call them $v_2, v_3, v_4, v_5$) have the other bit.
The $6$ triples containing $v_1$ are different colors from the other $4$
triples.
Indeed, the triples containing $v_1$ have $c_2
= i$, while the other triples have $c_2$ greater than $i$. 
Since $c_M$ is a $(4,3)$-coloring of graphs, $c_4$ tells us that the triples containing $v_1$ have 
to use at least $3$ colors. On the other hand,
by the coloring $c_1$, the triples in the 4-set $\{v_2,v_3,v_4,v_5\}$ 
have to use at least 3 colors. Hence, at least 6 colors have to 
be used on the set of five vertices.

\medskip

\noindent \textbf{Case 2}: Two of the vertices 
(call them $v_1,v_2$) have one bit at coordinate $i$, while the other three vertices 
(call them $v_3, v_4, v_5$) have the other
bit. Let $V_0=\{v_1,v_2\}$ and $V_1=\{v_3,v_4,v_5\}$.
Let $A$ be the set of colors of triples in $\{v_1,...,v_5\}$. We partition $A$
into $A_0, A_1, A_2$ as follows. For each $j \in \{0,1,2\}$,
let $A_j$ be the set consisting of the colors of triples
containing exactly $j$ vertices from $V_0$.
It follows from the colorings $c_2$ and $c_3$ that the three color sets $A_0,
A_1, A_2$ form a partition of $A$. Indeed, the color in $A_0$ has second
coordinate $c_2$ greater than $i$, while the colors in $A_1$ and $A_2$ have
second coordinate $c_2 = i$.
Furthermore, the colors in $A_1$ have third coordinate $c_3$ distinct from the
third coordinate $c_3$ of the colors in $A_2$. Note also that $|A_0| = 1$.

\medskip

\noindent \textbf{Case 2a}: $|A_2|=3$. 

Since the coloring 
$c_M$ is a $(3,2)$-coloring of graphs, $c_4$ implies that 
the triples containing $v_1$ whose other two
vertices are in $V_1$ receive at least $2$ colors. This implies that $|A_1| \geq 2$ and, 
therefore, the number of colors used is at least $|A_0|+|A_1|+|A_2| \geq 6$.

\medskip

\noindent \textbf{Case 2b}: $|A_2|=2$.

Suppose without loss of generality that $(v_1,v_2,v_3)$ and $(v_1,v_2,v_4)$
have the same color, which is different from the color of $(v_1,v_2,v_5)$.
As each $K_4^{(3)}$ uses at least $3$ colors in coloring $c_1$, $(v_1,v_3,v_4)$
and $(v_2,v_3,v_4)$ have different colors. Note that
$c_4(v_1, v_3, v_4) = c_4(v_2, v_3, v_4) = c_M(v_3, v_4)$. Since $c_M$
is a $(3,2)$-coloring of graphs, at least one of $c_M(v_3, v_5)$ or
$c_M(v_4, v_5)$ is different from $c_M(v_3, v_4)$. Suppose, without
loss of generality, that $c_M(v_3, v_5) \neq c_M(v_3, v_4)$. Since
$c$ is defined as the product of $c_1, \ldots, c_4$, 
we see that the color of $(v_1, v_3, v_5)$ is different from both that of
$(v_1, v_3, v_4)$ and $(v_2, v_3, v_4)$. Thus $|A_1| \geq 3$.
Then the number of colors used is at least $|A_0|+|A_1|+|A_2| \geq 6$.

\medskip

\noindent \textbf{Case 2c}: $|A_2|=1$.

This implies that the three edges $(v_1,v_2,v_j)$ for $j=3,4,5$ are of the same color.
First note that as in the previous case, there are at least two
different colors among $c_M(v_3, v_4)$,  $c_M(v_3, v_5)$ and $c_M(v_4, v_5)$.
Without loss of generality, suppose that $c_M(v_3, v_4) \neq c_M(v_3, v_5)$.
Since $c$ is defined as the product of $c_1, \ldots, c_4$, 
this implies that the set $A_1' = \{c(v_1, v_3, v_4), c(v_2, v_3, v_4)\}$ is 
disjoint from the set $A_1'' = \{c(v_1, v_3, v_5), c(v_2, v_3, v_5)\}$.
Now, by considering the coloring $c_1$, since all 
three edges $(v_1,v_2,v_j)$ for $j=3,4,5$ are of the same color, 
we see that $|A_1'| = 2$ and $|A_1''| = 2$. Hence 
$|A_1| \ge |A_1'| + |A_1''| = 4$.
Then the number of colors used is at least $|A_0|+|A_1|+|A_2| \geq 6$.

\section{A chromatic number version of the Erd\H{o}s--Gy\'arf\'as problem}
\label{sec:chi_eg}

\subsection{Bounds on $F_{\chi}(r,4,3)$}

In this subsection, we prove Theorem \ref{thm:chi_4_3}. This asserts that
\[ 2^{\log^2 r/36} \le F_{\chi}(r,4,3) \le C \cdot 2^{130 \sqrt{r \log r}}. \]

In order to obtain the upper bound, we use the concept of dense pairs.
Suppose that a graph $G$ is given. For positive
reals $\varepsilon$ and $d$, a pair of disjoint vertex subsets  
$(V_1, V_2)$ is \emph{($\varepsilon, d$)-dense} if
for every pair of subsets $U_1 \subseteq V_1$ and $U_2 \subseteq V_2$
satisfying $|U_1| \ge \varepsilon |V_1|$ and 
$|U_2| \ge \varepsilon |V_2|$,  we have
\[ e(U_1, U_2) \ge d |U_1| |U_2|,
\]
where $e(U_1, U_2)$ is the number of edges of $G$ with one
endpoint in $U_1$ and the other in $U_2$.
The following result is due to Peng, R\"odl and Ruci\'nski \cite{PRR}. 
Recall that the edge density of a graph $G$ with $m$ edges and $n$ vertices is
$m/\binom{n}{2}$.

\begin{THM} \label{thm:regularity}
For all positive reals $d$ and $\varepsilon$, 
every graph on $n$ vertices of edge density at least $d$ contains an
$(\varepsilon, d/2)$-dense pair $(V_1, V_2)$ for which
\[ |V_1| = |V_2| \ge \frac{1}{8}n d^{12/\varepsilon}. \]
\end{THM}

The original theorem of Peng, R\"odl and Ruci\'nski takes a bipartite graph
with $n$ vertices in each part and $dn^2$ edges as input and outputs an 
$(\varepsilon,d/2)$-dense pair with 
parts of size at least $\frac{1}{2}n d^{12/\varepsilon}$. 
The theorem as stated above is an immediate corollary since every $n$-vertex
graph of density $d$ contains a bipartite subgraph with $m = \lfloor \frac{n}{2} \rfloor \ge \frac{n}{4}$ vertices in each part and at least $d m^2$ edges. 

\begin{proofof}{upper bound in Theorem \ref{thm:chi_4_3}}
Let $n = F_{\chi}(r,4,3) - 1$ and
suppose that a chromatic-$(4,3)$-coloring of $K_n$ using $r$ colors is given.
Take a densest color, say red, and consider the graph $\mathcal{G}$ 
induced by the red edges. This graph has density at least $\frac{1}{r}$.
By applying Theorem \ref{thm:regularity} with
$\varepsilon = \left(\frac{\ln r}{r}\right)^{1/2}$,
we obtain an $(\varepsilon, \frac{1}{2r})$-dense pair $(V_1, V_2)$ in $\mathcal{G}$
such that
\[ |V_1| = |V_2| \ge \frac{1}{8}n\left(\frac{1}{r}\right)^{12/\varepsilon}
   \ge n e^{-13 \sqrt{r \ln r}}. \]

For a color $c$ which is not red, let $\mathcal{G}_{+c}$ be the graph obtained
by adding all edges of color $c$ to the graph $\mathcal{G}$. 
Since the given coloring is a chromatic-$(4,3)$-coloring, we see that 
$\mathcal{G}_{+c}$ is $3$-colorable for all $c$. 
Consider an arbitrary proper $3$-coloring of $\mathcal{G}_{+c}$.
If there exists a color class in this proper coloring 
which intersects both $V_1$ and $V_2$ in 
at least $\varepsilon|V_1|$ vertices, then, since $(V_1, V_2)$ is an
$(\varepsilon, \frac{1}{2r})$-dense pair, there exists an edge 
between the two intersections, thereby
contradicting the fact that the $3$-coloring is proper. 

Hence, $\mathcal{G}_{+c}$ has an independent set $I_c$ of size at least 
$(1-2\varepsilon)|V_1|$ in either $V_1$ or $V_2$. 
For $i=1,2$, define $C_i$ to be the set of colors $c \in [r]$ for which 
this independent set $I_c$ is in $V_i$. 
Since $|C_1| + |C_2| \ge r-1$, 
we may assume, without loss of generality, that $|C_1| \ge \frac{r-1}{2}$. 

For each $v \in V_1$, let $d(v)$ be the number of colors 
$c \in C_1$ for which $v \in I_c$. Note that
\begin{align}
\label{eq:degreesum}
  \sum_{v \in V_1} d(v) = \sum_{c \in C_1} |I_c| 
   \ge |C_1| \cdot (1-2\varepsilon)|V_1|.
\end{align}
For each $X \subset C_1$ of size $\frac{r}{4}$, let
$I_X = \bigcap_{c \in X} I_c$. We have
\begin{align*}
  \sum_{\substack{X \subset C_1 \\ |X| = r/4}} |I_X|
 = \sum_{v \in V_1} {d(v) \choose r/4}
 \ge |V_1| \cdot {|C_1| \cdot (1-2\varepsilon) \choose r/4},
\end{align*}
where the inequality follows from \eqref{eq:degreesum} and
convexity. Since $|C_1| \ge \frac{r-1}{2}$, we have
\[ 
  \sum_{\substack{X \subset C_1 \\ |X| = r/4}} |I_X|
 \ge (1-8\varepsilon)^{r/4} |V_1| \cdot {|C_1| \choose r/4}.
\]
Thus we can find a set $X \subset C_1$ for which
$|I_X| \ge (1-8\varepsilon)^{r/4}|V_1|$. By
definition, the set $I_X$ does not contain any color from $X$ 
and hence the original coloring induces a chromatic-$(4,3)$-coloring
of a complete graph on $|I_X|$ vertices using at most $3r/4$ colors.
This gives
\[
  F_{\chi}\left(\frac{3r}{4}, 4, 3\right) - 1
  \ge |I_X| 
  \ge (1-8\varepsilon)^{r/4} |V_1|.
\]
For $\varepsilon \le 1/16$, the inequality $1-8\varepsilon \ge e^{-16\varepsilon}$ holds.
Hence, for large enough $r$, the right-hand side above is at least
\[
  e^{-4 \varepsilon r} \cdot n e^{-13\sqrt{r \ln r}} 
  =
  e^{-4\sqrt{r\ln r}} \cdot n e^{-13\sqrt{r \ln r}}   
  = (F_{\chi}(r, 4, 3) - 1) e^{-17 \sqrt{r \ln r}}.
\]
We conclude that there exists $r_0$ such that if $r \ge r_0$, then
\[
F_{\chi}(r,4,3) \le  e^{17 \sqrt{r \ln r}} F_{\chi}\left(\frac{3r}{4}, 4, 3\right).
\]

We now prove by induction that there is a constant $C$ such that $F_{\chi}(r,4,3) \leq C e^{130 \sqrt{r \ln r}}$ holds for all $r$. This clearly holds for the base cases $r<r_0$, so suppose $r \geq r_0$. Using the above inequality and the induction hypothesis, we obtain
\[
F_{\chi}(r,4,3) \le  e^{17 \sqrt{r \ln r}} F_{\chi}\left(\frac{3r}{4}, 4, 3\right) \le 
e^{17 \sqrt{r \ln r}} Ce^{130 \sqrt{(3r/4)\ln (3r/4)}} \le Ce^{130\sqrt{r\ln r}},\]
which completes the proof. 
\end{proofof}

We now turn to the proof of the lower bound.
In order to establish the lower bound, 
we show that Mubayi's coloring $c_M$ is in fact a chromatic-$(4,3)$-coloring.
This then implies that $F_{\chi}(r,4,3) \ge 2^{\log^2 r/ 36}$, as claimed.
Recall that in the coloring $c_M$, we view the vertex set of $K_n$
as a subset of $[m]^t$ for some integers $m$ and $t$ and, for two vertices
$x, y \in [m]^t$ of the form $x = (x_1, \ldots, x_t)$ and
$y = (y_1, \ldots, y_t)$, we let 
\[ c_M(x,y) = \Big(\{x_i, y_i\}, a_1, \ldots, a_t\Big), \]
where $i$ is the minimum index for which $x_i \neq y_i$ and
$a_j = \delta(x_j, y_j)$ is the Dirac delta function.

\begin{proofof}{lower bound in Theorem \ref{thm:chi_4_3}}
Consider the coloring $c_M$ on the vertex set $[m]^t$.
Suppose that two colors $c_{1}$ and $c_{2}$ are given and let
\[ c_1 = \Big(\{x_{1}, y_{1}\}, a_{1,1}, \ldots, a_{1,t} \Big)
 \quad \textrm{and} \quad 
  c_2 = \Big(\{x_{2}, y_{2}\}, a_{2,1}, \ldots, a_{2,t} \Big). \]
Suppose that $a_{1, i_1}$ is the first non-zero $a_{1,j}$ term 
and $a_{2, i_2}$ is the first non-zero $a_{2,j}$ term. In other 
words, for a pair of vertices which are colored by $c_1$, the
first coordinate in which the pair differ is the $i_1$-th coordinate
(and a similar claim holds for $c_2$).

Let $\mathcal{G}$ be the graph induced by
the edges which are colored by either $c_1$ or $c_2$.
We will prove that $\chi(\mathcal{G}) \le 3$ by presenting
a proper vertex coloring of $\mathcal{G}$ using three colors, red, blue
and green.

\medskip
\noindent \textbf{Case 1}: $i_{1}=i_{2}=i$ for some index $i$. 

First, color all the vertices whose $i$-th coordinate is equal
to $x_{1}$ in red. Second, color all the vertices
whose $i$-th coordinate is equal to $x_{2}$ in blue
(if $x_1 = x_2$, there are no vertices of color blue).
Third, color all other vertices in green.

To show that this is a proper coloring, note that
if the color between two vertices $z, w \in [m]^t$ is either
$c_1$ or $c_2$, then the $i$-th coordinate of 
$z$ and $w$ must be different.
This shows that the set of red vertices and
the set of blue vertices are both independent sets. It
remains to show that the set of green vertices is an 
independent set. To see this, note that if the color between $z$ and
$w$ is either $c_1$ or $c_2$, then the $i$-th coordinates 
$z_i$ and $w_i$ must satisfy
\[ \{z_i, w_i\} = \{x_{1}, y_{1}\} \quad \textrm{or} \quad \{x_{2}, y_{2}\}, \]
as this is the only way the first coordinate of $c_M(z, w)$
can match that of $c_1$ or $c_2$. However, all vertices
which have $i$-th coordinate $x_{1}$ or $x_{2}$ are
excluded from the set of green vertices. This shows that our 
coloring is proper. 


\medskip
\noindent \textbf{Case 2}: $i_{1}\neq i_{2}$.

Without loss of generality, we may assume that $i_1 < i_2$.
We will find a proper coloring by considering only
the $i_1$-th and $i_2$-th coordinates.
For $v \in [m]^t$ of the form $v = (v_1, v_2, \ldots, v_t)$, 
let
\[
  \pi_{i_1}(v) = 
  \begin{cases}
  0 & \text{if } v_{i_1} = x_1 \\
  1 & \text{if } v_{i_1} = y_1 \\
  * & \text{otherwise}\\
  \end{cases} 
  \qquad \text{and} \qquad
  \pi_{i_2}(v) = 
  \begin{cases}
  0 & \text{if } v_{i_2} = x_2 \\
  1 & \text{if } v_{i_2} = y_2 \\
  * & \text{otherwise}\\
  \end{cases}.
\]
Consider the projection map 
\[ \pi \,:\, [m]^t \rightarrow \{0, 1, *\} \times \{0, 1, *\} \]
defined by $\pi(v) = (\pi_{i_1}(v), \pi_{i_2}(v))$ and let 
$\mathcal{H} = \pi(\mathcal{G})$ be the graph on $\{0, 1, *\} \times \{0, 1, *\}$
induced by the graph $\mathcal{G}$ and the map $\pi$. More precisely,
a pair of vertices $v, w \in \{0, 1, *\} \times \{0, 1, *\}$
forms an edge if and only if there exists an edge of $\mathcal{G}$
between the two sets $\pi^{-1}(v)$ and $\pi^{-1}(w)$
(see Figure \ref{fig:proj_chi_4_3}). Note that a proper coloring
of $\mathcal{H}$ can be pulled back via $\pi^{-1}$ to give a 
proper coloring of $\mathcal{G}$. It therefore suffices to find a proper
$3$-coloring of $\mathcal{H}$.

\begin{figure}[htp]
\centering
\includegraphics[scale=0.75]{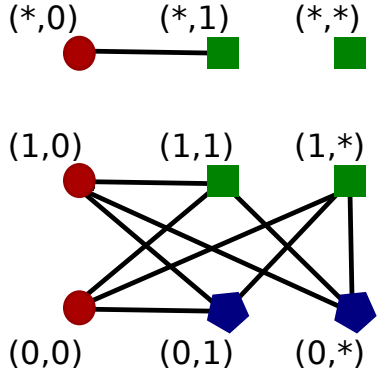}
\caption{Graph $\mathcal{H} = \pi(\mathcal{G})$ when $a_{1,i_2} = 1$.}
\label{fig:proj_chi_4_3}
\end{figure}

Consider two vertices $z,w \in [m]^t$. If $c_M(z,w) = c_2$, then
the first coordinate in which $z$ and $w$ differ is the 
$i_2$-th coordinate. This implies that $z$ and $w$ have 
identical $i_1$-th coordinate. Hence, the set of 
possible edges of the form $\{\pi(z), \pi(w)\}$ is 
$E_2 = \left\{ \{00, 01\}, \{10, 11\}, \{*0, *1\} \right\}$. 

Now suppose that $c_M(z,w) = c_1$. Then the possible edges
of the form $\{\pi(z), \pi(w)\}$ differ according to the
value of $a_{1, i_2}$.

\noindent \textbf{Case 2a}: $a_{1, i_2} = 0$. 

In this case, the $i_2$-th coordinate of $z$ and $w$ must be
the same and thus the possible edges of the form $(\pi(z), \pi(w))$
are $E_1 = \left\{ \{00, 10\}, \{01, 11\}, \{0*, 1*\} \right\}$.
One can easily check that the graph with edge set $E_1 \cup E_2$ 
is bipartite.

\noindent \textbf{Case 2b}: $a_{1, i_2} = 1$. 

In this case, the $i_2$-th coordinate of $z$ and $w$ must be
different and thus the possible edges of the form $(\pi(z), \pi(w))$
are $E_1 = \left\{ \{00, 11\}, \{00, 1*\}, \{01, 10\}, \{01, 1*\}
, \{0*, 10\}, \{0*, 11\}, \{0*, 1*\} \right\}$. A $3$-coloring
of the graph with edge set $E_1 \cup E_2$ is given by coloring the set of vertices
$\{00, 10, *0\}$ in red, $\{01, 0*\}$ in blue and $\{11, 1*, *1, **\}$ in green
(see Figure \ref{fig:proj_chi_4_3}).
\end{proofof}

\subsection{An edge partition with slowly growing chromatic number}

In this section, we will prove Theorem \ref{thm:chi_slow_grow} 
by showing that $c_M$ has the required property.

\medskip

\noindent {\bf Theorem \ref{thm:chi_slow_grow}}.
The coloring $c_M$ has the following property: for every
subset $X$ of colors with $|X| \geq 2$, the subgraph induced by the edges colored
with a color from $X$ has chromatic number at most $2^{3 \sqrt{|X| \log |X|}}$.

\medskip

\begin{proof}
Consider the coloring $c_M$ on the vertex set $[m]^t$.
For a set of colors $X$, let $\mathcal{G}_X$  be the graph 
induced by the edges colored by a color from the set $X$.
Recall that each color $c$ under this coloring is of the form
\begin{align*}
c  &= \Big(\{v_i, w_i\}, a_1, a_2, \ldots, a_t \Big).
\end{align*}
Define $\iota(c) = i$ to be the minimum index $i$ for which 
$a_i = 1$. For this index $i$, define $\eta_1(c) = v_i$ and $\eta_2(c) = w_i$,
where we break symmetry by imposing $v_i < w_i$. Let 
$a_j(c) = a_j$ for $j = 1,\ldots, t$. 

Given a set of colors $X$, construct an auxiliary graph $\mathcal{H}$ over the vertex set $X$ whose
edges are defined as follows. For two colors $c_1, c_2 \in X$,
let $i_1 = \iota(c_1)$, $i_2 = \iota(c_2)$ and assume
that $i_1 \le i_2$. Then $c_1$ and $c_2$ are adjacent
if and only if $a_{i_2}(c_1) = 1$ (it is well-defined 
since if $i_1 = i_2$, then $a_{i_2}(c_1) = a_{i_1}(c_2) = 1$).
Let $\mathcal{I}$ be the family of all independent sets in $\mathcal{H}$.
We make the following claim, whose proof will be given later.

\begin{CLAIM} \label{claim:indep_sets}
The following holds: \\
(i) For all $I \in \mathcal{I}$, the graph $\mathcal{G}_I$ is bipartite. \\
(ii) $\chi(\mathcal{G}_X) \le |\mathcal{I}|$.
\end{CLAIM}

Suppose that the claim is true. Based on this claim, we will prove by induction on $|X|$ that
$\chi(\mathcal{G}_X) \le 2^{3 \sqrt{|X| \log |X|}}$. 
For $|X|=2$, we proved in the previous subsection that $c_M$ is a chromatic-$(4,3)$-coloring, that is,
the union of any two color classes is $3$-colorable. This clearly implies the required result in this case.
Now suppose that the statement has been established for all
sets of size less than $|X|$.

Let $\alpha = \left\lceil \sqrt{\frac{|X|}{\log |X|}} \right\rceil$. If there
exists an independent set $I \in \mathcal{I}$ of
size at least $\alpha$, then, by the fact that
$\mathcal{G}_X = \mathcal{G}_I \cup \mathcal{G}_{X \setminus I}$ 
and Claim \ref{claim:indep_sets} (i),
we have
\[ \chi(\mathcal{G}_X) \le \chi(\mathcal{G}_I) \cdot \chi(\mathcal{G}_{X \setminus I})
           \le 2 \chi(\mathcal{G}_{X \setminus I}).
\]
If $|X \setminus I| \ge 2$, then the right hand side is at most $2 \cdot 2^{3\sqrt{|X\setminus I| \log |X\setminus I|}} < 2^{3\sqrt{|X| \log |X|}}$ (the inequality comes from $|I| \ge \alpha$) by the inductive hypothesis, and if $|X \setminus I| \le 1$, then since $\chi(\mathcal{G}_{X \setminus I}) \leq 2$, the right hand side is at most $4$. Hence the claimed bound holds in both cases.

On the other hand, if the
independence number is less than $\alpha$, then, 
by Claim \ref{claim:indep_sets} (ii) and the fact that $|X| \ge 2$, we have
\[ \chi(\mathcal{G}_X) \le \sum_{i=0}^{\alpha - 1} {|X| \choose i}
           \le |X|^{2\sqrt{|X|/\log |X|}}
           = 2^{2\sqrt{|X|\log |X|}}. \]
This proves the theorem up to Claim \ref{claim:indep_sets}, which we now consider.
\end{proof}

\begin{proofof}{Claim \ref{claim:indep_sets}}
(i) Suppose that $I \in \mathcal{I}$ is given. By definition, for each color $c \in I$, 
we have distinct values of $\iota(c)$. 
For each $c \in I$, consider the map $\pi_c : [m]^t \rightarrow \{0, 1\}$, 
where for $x \in [m]^t$ of the form $x = (x_1, x_2, \ldots, x_t)$, we define
\[ 
\pi_c(x) = 
	\begin{cases}
	0 & \text{if } x_{\iota(c)} \le \eta_1(c) \\
	1 & \text{if } x_{\iota(c)} > \eta_1(c).
	\end{cases}
\]
Define the map $\pi : [m]^t \rightarrow \{0, 1\}^{I}$ as
\[ \pi(x) = (\pi_c(x))_{c \in I}. \]
Consider the graph $\pi(\mathcal{G}_I)$ over the vertex set $\{0, 1\}^{I}$.
Let $c$ and $c'$ be two distinct colors in $I$.
If $\iota(c') < \iota(c)$, then
$a_{\iota(c')}(c) = 0$ since $\iota(c)$ is the minimum index $i$ 
for which $a_i(c) = 1$ and if $\iota(c') > \iota(c)$, then
$a_{\iota(c')}(c) = 0$ since $I$ is an independent set 
in the auxiliary graph $\mathcal{H}$ defined above.
Hence, if $e = \{y,z\}$ is an edge of color $c$, then the
two vectors $y$ and $z$ have identical $\iota(c')$-coordinate for all $c' \neq c$, thus implying that $\pi(y)$ and $\pi(z)$ have identical $c'$-coordinate for all $c' \neq c$.
Further note that for $x \in \{0,1\}^I$, we have $\pi_c(x) = 0$ if the $\iota(c)$-th coordinate of $x$ is $\eta_1(c)$ and 
$\pi_c(x)=1$ if the $\iota(c)$-th coordinate of $x$ is $\eta_2(c)$.
Since $\{y_{\iota(c)}, z_{\iota(c)}\} = \{\eta_1(c), \eta_2(c)\}$, we see that $\pi_c(y) \neq \pi_c(z)$.

Therefore, two vertices $v,w \in \{0, 1\}^{I}$ can be adjacent in $\pi(\mathcal{G}_I)$
if and only if $v$ and $w$ differ in exactly one coordinate, implying
that $\pi(\mathcal{G}_I)$ is a subgraph of the hypercube, which is clearly bipartite.
A bipartite coloring of this graph can be pulled back to give
a bipartite coloring of $\mathcal{G}_I$.  

\medskip

\noindent (ii) We prove this by induction on the size of the set $X$. 
The claim is trivially true for $|X| = 0$ and $1$, since $|\mathcal{I}| = 1$ and
$2$, respectively, and the graph $\mathcal{G}_X$ has chromatic number 
$1$ and $2$, respectively.

Now suppose that we are given a set $X$ and the family 
$\mathcal{I}$ of independent sets in $\mathcal{H}$ (as defined above). 
Let $c \in X$ be a color with maximum $\iota(c)$
and let $i = \iota(c)$.
Let $\mathcal{I}_c$ be the family of independent sets containing $c$ 
and $\mathcal{I}'_c$ be the family of all other independent sets. 

Let $A$ be the subset of vertices of $[m]^t$ whose $i$-th 
coordinate is $\eta_1(c)$. For two vectors $x, y \in A$, 
we have $a_i( c_M(x,y) ) = 0$, since both $x$ and $y$ 
have $i$-th coordinate $\eta_1(c)$. 
Hence, in the subgraph of $\mathcal{G}_X$ induced
on the set $A$, we only see colors $c' \in X$ 
which have $a_{i}(c') = 0$. 
Let $X_c \subseteq X$ be the set of colors $c'$ such that $a_{i}(c') = 0$.
The observation above implies that $\mathcal{G}_X[A]$ is a subgraph
of $\mathcal{G}_{X_c}$. By the inductive hypothesis, $\chi(\mathcal{G}_{X_c})$
is at most the number of independent sets of $\mathcal{H}[X_c]$. 
Moreover, by the definitions of $X_c$ and $\mathcal{I}_c$ and the
choice of $c$, the independent sets of 
$\mathcal{H}[X_c]$ are in one-to-one correspondence with
the independent sets in $\mathcal{I}_c$. 
Thus, we have
\[ 
	\chi(\mathcal{G}_X[A]) 
	\le \chi(\mathcal{G}_{X_c}) 
	\le |\mathcal{I}_c|. 
\]

Now consider the set $B = [m]^t \setminus A$. The subgraph
of $\mathcal{G}_X$ induced on $B$ does not contain any edge of color $c$
and therefore $\mathcal{G}_{X}[B]$ is a subgraph of $\mathcal{G}_{X \setminus \{c\}}$.
By the inductive hypothesis, $\chi(\mathcal{G}_{X \setminus \{c\}})$
is at most the number of independent sets of $\mathcal{H}[X \setminus \{c\}]$.
By definition, the independent sets of $\mathcal{H}[X \setminus \{c\}]$ are in 
one-to-one correspondence with independent sets in $\mathcal{I}'_c$.
Therefore, we have
\[ \chi(\mathcal{G}_X[B]) \le \chi(\mathcal{G}_{X \setminus \{c\}}) \le |\mathcal{I}'_c|. \]
Hence,
\[ \chi(\mathcal{G}_X) \le \chi(\mathcal{G}_X[A]) + \chi(\mathcal{G}_X[B]) 
   \le |\mathcal{I}_c| + |\mathcal{I}'_c| 
   = |\mathcal{I}|,
\]
and the claim follows.
\end{proofof}

Using Theorem \ref{thm:chi_slow_grow}, we can now prove Theorem \ref{thm:chi_slow_grow_less_colors}, which
we restate here for the reader's convenience.
Recall that for an edge partition $E_1 \cup \ldots \cup E_t$ of the complete graph $K_n$
and a set $I \subseteq [t]$,
we define $\mathcal{G}_I$ as the subgraph of $K_n$ with edge set $\bigcup_{i \in I} E_i$. 

\medskip
\noindent {\bf Theorem \ref{thm:chi_slow_grow_less_colors}}.
There exists a positive real $r_0$ such that the following holds
for every positive integer $r$ and positive real $\alpha \le 1$
satisfying $(\log r)^\alpha \ge r_0$.
For $n = 2^{(\log r)^{2 + \alpha}/200}$,
there exists a partition $E = E_1 \cup \dots \cup E_{\sqrt{r}}$ of
the edge set of the complete graph $K_n$ such that
\[ \chi(\mathcal{G}_I) \le 2^{3(\log r)^{\alpha/2}\sqrt{|I| \log 2 |I|}} \]
for all $I \subset [\sqrt{r}]$.
\medskip

\begin{proof}
Let $N = 2^{\log^2 r/200}$ and $t = (\log r)^{\alpha}$ (since 
$(\log r)^{\alpha} \ge r_0$ and $\alpha \le 1$, we can guarantee
that $N$ and $t$ are large enough by asking that $r_0$ be
large enough). Color the
edge set of the complete graph on the vertex set $[N]^t$ as follows.
For two vectors $v, w \in [N]^t$ of the form $v = (v_1 ,\ldots,v_t)$ 
and $w = (w_1, \ldots, w_t)$, we let
\[ c(v,w) = \Big(i, c_M(v_i, w_i)\Big), \]
where $i$ is the minimum index for which $v_i \neq w_i$. Since
$c_M$ on $K_N$ uses at most
$2^{6\sqrt{\log N}} \le \frac{\sqrt{r}}{\log r}$ colors (see the
discussion in Section \ref{sec:preliminaries}), our coloring
uses at most 
\[ t \cdot \frac{\sqrt{r}}{\log r} \le \sqrt{r} \] 
colors in total. Since $n = N^t$, this 
coloring gives an edge partition $E = E_1 \cup \dots \cup E_{s}$
of the complete graph on $n$ vertices, for some integer $s \le \sqrt{r}$.

Now suppose that a set $I \subset [s]$ is given. The set $I$ can be
partitioned into $t$ sets $I_1 \cup \dots \cup I_t$
according to the value of the first coordinate as follows:
for each $i \in [t]$, define $I_i$ as the set of indices $j \in I$ for which 
the color of the edges $E_j$ has $i$ as its first coordinate.
For each $i$, let $\pi_i \,:\, [N]^t \rightarrow [N]$ be the
projection map to the $i$-th coordinate. Then the graph
$\pi_i(\mathcal{G}_{I_i})$ becomes a subgraph of $K_N$
induced by the union of $|I_i|$ colors of $c_M$. 
Hence, by Theorem \ref{thm:chi_slow_grow}, we know that 
\[ \chi(\mathcal{G}_{I_i}) \le \chi(\pi_i(\mathcal{G}_{I_i})) \le 2^{3\sqrt{|I_i| \log 2 |I_i|}} \]
for each $i \in [t]$, where we introduce the extra $2$ in the logarithm to account for the possibility that $|I_i| = 1$. Therefore, we see that
\[ \chi(\mathcal{G}_I) \le \prod_{i \in [t]} \chi(\mathcal{G}_{I_i}) \le  2^{3\sum_{i \in [t]} \sqrt{|I_i| \log 2 |I_i|}}. \]
Since $\sqrt{x \log 2 x}$ is concave, Jensen's inequality implies that the sum in the exponent satisfies
\[\sum_{i \in [t]} \sqrt{|I_i| \log 2 |I_i|} \leq t \sqrt{(|I|/t) \log (2 |I|/t)}
\le \sqrt{t |I| \log 2 |I|} =  (\log r)^{\alpha/2} \sqrt{|I| \log 2 |I|} . \]
This implies the required result.
\end{proof}

\section{Concluding Remarks}
\label{sec:conclusion}

\subsection{The grid Ramsey problem with asymmetric colorings}

One may also consider an asymmetric version of the grid Ramsey problem,
where we color the row edges using $r$ colors but are allowed to 
use only two colors on the column edges. Let $G(r,2)$ be the minimum $n$ 
for which such a coloring is guaranteed to contain an alternating rectangle. One can easily 
see that
\[ r \le G(r,2) \le r^3 + 1. \]
The following construction improves the lower bound to 
$G(r,2) \ge \frac{1}{4}r^2$. Let $n = \frac{1}{4} r^2$ and $p$ be a prime satisfying 
$\frac{r}{2} \le p \le r$ and $n \le 2^p$ (the existence of such a prime follows from
Bertrand's postulate). 
Consider the $n \times n$ grid. 
For each $i \in [n]$, assign to the $i$-th row a sequence 
$(a_{i,1}, \ldots, a_{i, p}) \in [r]^{p}$ so that
for all distinct $i$ and $i'$ there exists at most one coordinate
$j \in [p]$ for which $a_{i,j} = a_{i',j}$ (the construction will be
given below). Given these sequences, for each $i \in [n]$ and distinct
$j, j' \in [n]$, color the edge $\{(i,j), (i,j')\}$ as follows:
examine the binary expansions of $j$ and $j'$ to identify the
first bit $t$ in which the two differ and color the edge with color $a_{i,t}$
(this is possible since $2^p \ge n$). 
For two distinct rows, suppose that the sequences corresponding to these rows 
coincide in the $k$-th coordinate. Then the intersection of the two rows 
is a subgraph of the graph connecting vertices
whose $k$-th bit in the binary expansion is 0 to those whose $k$-th bit
in the binary expansion is 1. Thus, the intersection of any two rows
is a bipartite graph and therefore, by the same argument as in the proof of Lemma \ref{lem:row_chromatic}, 
we obtain $G(r,2) \ge n$ 
(note that we in fact obtain a coloring of the $cr^2 \times 2^{c'r}$ grid).

It suffices to construct a collection of sequences with the property claimed above. 
For $a,b \in \mathbb{Z}_p$, consider the following sequence
with entries in $\mathbb{Z}_p$:
\[
	B_{a,b} = \Big( a, a + b, a + 2b, \ldots, a + (p-1)b \Big).
\]
For two distinct pairs $(a,b)$ and $(a',b')$,
the sequences $B_{a,b}$ and $B_{a',b'}$ can overlap in the $i$-th coordinate
if and only if $a + ib = a' + ib'$, which is 
equivalent to $(b-b')i = a' - a$. Since $(a,b) \neq (a',b')$, we see
that there exists at most one index $i$ in the range $0 \le i \le p-1$
for which $a + ib = a' + ib'$. Thus the sequences $B_{a,b}$ have the claimed property.
Note that the total number of sequences is at least $p^2 \ge n$. 
Moreover, since $p \le r$, by abusing
notation, we may assume that the sequences are in fact in $[r]^p$ and,
therefore, we can use them in the construction of our coloring.

The following question may be more approachable than the corresponding problem for $G(r)$.

\begin{QUES}
Can we improve the upper bound on $G(r, 2)$?
\end{QUES}

\subsection{The Erd\H{o}s--Gy\'arf\'as problem in hypergraphs}

As mentioned in the introduction,
for each fixed $i$ with $0 \le i \leq k$ and large enough $p$, 
\[ F_k\left(r, p, {p-i \choose k-i} + 1\right) \le r^{r^{\iddots^{r^{c_{k,p}}} }}, \]
where the number of $r$'s in the tower is $i$.
It would be interesting to establish a lower bound on $F(r, p, {p-i \choose k-i})$ 
exhibiting a different behavior.

\begin{PROB}
Let $p, k$ and $i$ be positive integers with $k \ge 3$ and $0 < i < k$.
Establish a lower bound on $F_k(r, p, {p-i \choose k-i})$ 
that is significantly larger than the upper bound on $F_k(r,p,{p-i \choose k-i} + 1)$
given above.
\end{PROB}

We have principally considered the $i=1$ case of this question. For example, the Erd\H{o}s--Gy\'arf\'as 
problem on whether $F(n,p,p-1)$ is superpolynomial for all $p \geq 3$ corresponds to the case where
$k = 2$ and $i = 1$. Theorems~\ref{thm:grid_main} and \ref{thm:F_3_5_6} represent progress on the analogous problem
with $k = 3$. The next open case, showing that $F_3(r, 6, 10)$ is superpolynomial, appears difficult.

For $i \geq 2$, it seems likely that one would have to invoke a variant of the stepping-up technique of Erd\H{o}s and
Hajnal (see, for example, \cite{GrRoSp}). In particular, we would like to know the answer to the following question.

\begin{QUES}
Is $F_3(r, p, p-2)$ larger than $2^{r^c}$ for any fixed $c$?
\end{QUES}

For $p = 4$, a positive solution to this problem follows since we know that the Ramsey number of $K_4^{(3)}$ is 
double exponential in the number of colors (see, for example, \cite{AxGyLiMu}). The general case appears to be much more difficult.

Another case of particular importance is $F_{2d-1}(r, 2d, d+1)$, since it is this function (or rather a $d$-partite variant)
which is used by Shelah in his proof of the Hales--Jewett theorem. If the growth rate of this function is a tower of 
bounded height for all $d$, then it would be possible to give a tower-type bound for Hales--Jewett numbers. However, we expect
that this is not the case.

\begin{PROB}
Show that for all $s$ there exists $d$ such that 
\[ F_{2d-1}\left(r, 2d, d+1\right) \ge 2^{2^{\iddots^{2^{r}}}}, \]
where the number of $2$'s in the tower is at least $s$.
\end{PROB}

\subsection{Studying the chromatic number version of the Erd\H{o}s--Gy\'arf\'as problem}

Since we know that both $F(r,p,p-1)$ and $F_\chi(r,4,3)$ are superpolynomial in $r$, it is natural to ask the following 
question (see also \cite{CoFoLeSu}).

\begin{QUES} \label{que:chi_poly}
Is $F_{\chi}(r,p,p-1)$ superpolynomial in $r$?
\end{QUES}

By following a similar line of argument to the lower bound for $F_\chi(r,4,3)$,
we can show that $c_M$ is also a chromatic-$(5,4)$-coloring.
Therefore, $F_{\chi}(r,5,4) = 2^{\Omega(\log^2 r)}$, answering Question~\ref{que:chi_poly} for $p =5$. Since the proof is
based on rather tedious case analysis, we will post a supplementary note rather than including the details here. It would
be interesting to determine whether the $(p, p-1)$-colorings defined in \cite{CoFoLeSu} are also chromatic-$(p,p-1)$-colorings. If so, they 
would provide a positive answer to Question~\ref{que:chi_poly}. 



In Theorem~\ref{thm:chi_4_3}, we showed that 
$2^{\Omega(\log^2 r)} \le F_{\chi}(r, 4, 3) \le 2^{O(\sqrt{r \log r})}$.
It would be interesting to reduce the gap between the lower and upper bounds.
Since $F_{\chi}(r,4,2) \ge 2^{r} + 1$, we see that
$F_{\chi}(r,4,2)$ is exponential in $r$, while $F_{\chi}(r,4,3)$ is subexponential in $r$.
For $p \ge 5$, the value of $q$ for which the transition
from exponential to subexponential happens is not known. However, recall that 
$F_{\chi}(r, 2^d + 1, d+1)$ is exponential in $r$ for all $d \ge 1$.
This followed from showing that in the edge coloring $c_B$
the union of every $d$ color classes induces a graph
of chromatic number $2^d$ (see Section~\ref{sec:preliminaries}). 
The following question asks whether
a similar edge coloring exists if we want the union of every
$d$ color classes to induce a graph of chromatic number at most $2^d - 1$.

\begin{QUES} \label{ques:ch_3}
Is $F_\chi(r,2^{d}, d+1) = 2^{o(r)}$ for all $d \ge 2$?
\end{QUES}

A positive answer to Question \ref{ques:ch_3} would allow us to determine,
for all $p$, the maximum value of $q$ for which $F_{\chi}(r,p,q)$
is exponential in $r$. Indeed, for $2^{d-1} < p \le 2^{d}$, we have
\[
	F_\chi(r, p, d) \ge F_\chi(r, 2^{d-1}+1, d) = 2^{\Omega(r)},
\]
while a positive answer to Question \ref{ques:ch_3} would imply
\[
	F_\chi(r, p, d+1) \le F_\chi(r, 2^{d}, d+1) = 2^{o(r)}.
\] 
Hence, given a positive answer to Question \ref{ques:ch_3}, the maximum value of $q$ for
which $F_\chi(r, p, q)$ is exponential in $r$ would be $q = \lceil \log p \rceil$.

A key component in our proof of Theorem~\ref{thm:grid_main} was Theorem~\ref{thm:chi_slow_grow}, which
says that in the coloring $c_M$, the chromatic number of the union of any $s$ color classes is not too large. 
We suspect that our estimate on the chromatic number is rather weak. It would be interesting to improve it further.
More generally, we have the following rather informal question, progress on which might allow us to improve
the bounds in Theorem~\ref{thm:grid_main}.

\begin{QUES}
Given an edge partition of the complete graph $K_n$, how slowly can the chromatic
number of the graph determined by the union of $s$ color classes grow?
\end{QUES}

Finally, let $\mathcal{F}$ be a family of graphs and define $F(r,q; \mathcal{F})$
to be the minimum integer $n$ for which every edge coloring of $K_n$ with $r$ colors
contains a subgraph $F \in \mathcal{F}$ that contains fewer than
$q$ colors. $F(r,q; \mathcal{F})$ generalizes both $F(r,p,q)$ 
and $F_\chi(r,p,q)$ since we may take $\mathcal{F}$ to be
$\{K_p\}$ for $F(r,p,q)$ and the family of all $p$-chromatic graphs
for $F_\chi(r,p,q)$. Our results suggest that $F(r,q; \mathcal{F})$ is 
closely related to the chromatic number of the graphs in $\mathcal{F}$.
The case where $\mathcal{F}$ consists of a single complete bipartite graph was studied in \cite{AxFuMu}.




\medskip

\noindent
{\bf Acknowledgement.} We would like to thank the anonymous referee for several helpful comments.


\begin{thebibliography}{99}

\bibitem{AxFuMu}
M. Axenovich, Z. F\"uredi and D. Mubayi,
On generalized Ramsey theory: the bipartite case,
{\em J. Combin. Theory Ser. B} \textbf{79} (2000), 66--86.

\bibitem{AxGyLiMu}
{M. Axenovich, A. Gy\'arf\'as, H. Liu and D. Mubayi,} {Multicolor Ramsey numbers for triple systems,}
{\it Discrete Math.} {\bf 322} (2014), 69--77.

\bibitem{CoFoLeSu}
D. Conlon, J. Fox, C. Lee and B. Sudakov, 
The Erd\H{o}s--Gy\'arf\'as problem on generalized Ramsey numbers, to appear in {\em Proc. London Math. Soc.}

\bibitem{EiMu}
D. Eichhorn and D. Mubayi,
Edge-coloring cliques with many colors on subcliques,
{\em Combinatorica} \textbf{20} (2000), 441--444.

\bibitem{E75}
{P. Erd\H{o}s,} {Problems and results on finite and infinite graphs,} {in Recent advances in graph theory (Proc. Second Czechoslovak Sympos., Prague, 1974),} 183--192, Academia, Prague, 1975. 

\bibitem{E81}
{P. Erd\H{o}s,} {Solved and unsolved problems in combinatorics and combinatorial number theory,}
in {Proceedings of the twelfth southeastern conference on combinatorics, graph theory and computing, Vol. I (Baton Rouge, La., 1981),}
{\it Congr. Numer.} {\bf 32} (1981), 49--62. 

\bibitem{ErGy}
P. Erd\H{o}s and A. Gy\'arf\'as, 
A variant of the classical Ramsey problem, 
{\em Combinatorica} \textbf{17} (1997), 459--467.

\bibitem{FaGySz}
R. Faudree, A. Gy\'arf\'as and T. Sz\H{o}nyi,
Projective spaces and colorings of $K_m \times K_n$,
in Sets, graphs and numbers (Budapest, 1991), 273--278, 
Colloq. Math. Soc. J\'anos Bolyai, 60, North-Holland, Amsterdam, 1992. 

\bibitem{FoSu}
J. Fox and B. Sudakov,
Ramsey-type problem for an almost monochromatic $K_4$,
{\em SIAM J. Discrete Math.} {\bf 23} (2008), 155--162.

\bibitem{G01}
{W. T. Gowers,} {A new proof of Szemer\'edi's theorem,} {\it Geom. Funct. Anal.} {\bf 11} (2001), 465--588.

\bibitem{GrRoSp}
R. L. Graham, B. L. Rothschild and J. H. Spencer, 
{\bf Ramsey theory},
second ed., Wiley-Interscience Series in Discrete Mathematics and Optimization, John Wiley \& Sons Inc., 
New York, 1990.

\bibitem{Gyarfas}
A. Gy\'arf\'as,
On a Ramsey type problem of Shelah,
in Extremal problems for finite sets (Visegr\'ad, 1991), 283--287, 
Bolyai Soc. Math. Stud., 3, J\'anos Bolyai Math. Soc., Budapest, 1994.

\bibitem{HJ63}
{A. W. Hales and R. I. Jewett,} {Regularity and positional games,} {\it Trans. Amer. Math. Soc.} {\bf 106} (1963), 222--229.

\bibitem{H90}
{K. Heinrich,} {Coloring the edges of $K_m \times K_m$,} 
{\it J. Graph Theory} {\bf 14} (1990), 575--583.


\bibitem{KoMu}
A. Kostochka and D. Mubayi, 
When is an almost monochromatic $K_4$ guaranteed?, 
{\em Combin. Probab. Comput.} \textbf{17} (2008),
823--830.

\bibitem{Mubayi}
D. Mubayi, Edge-coloring cliques with three colors on all 4-cliques,
{\em Combinatorica} \textbf{18} (1998), 293--296.

\bibitem{PRR} 
Y. Peng, V. R\"odl and A. Ruci\'nski, Holes in graphs, {\em Electron. J. Combin.} \textbf{9} (2002), Research Paper 1, 18 pp.

\bibitem{Shelah}
S. Shelah, Primitive recursive bounds for van der Waerden numbers,
{\em J. Amer. Math. Soc.} \textbf{1} (1989), 683--697.

\end{thebibliography}
\end{document}